\documentclass[reqno]{amsart}
\usepackage[colorlinks=true, linkcolor=, citecolor=blue]{hyperref}
\usepackage{graphicx,subcaption}
\usepackage[utf8]{inputenc}
\usepackage[backend=bibtex]{biblatex}
\usepackage{amssymb}
\usepackage{amsmath}
\usepackage{graphicx}
\usepackage{amsthm,amsmath}
\usepackage{amsfonts}
\usepackage{graphicx}
\usepackage{enumitem}
\usepackage{bm}
\usepackage[utf8]{inputenc}
\numberwithin{equation}{section}
\newtheorem{theo}{Theorem}[section]

\newtheorem{lemma}[theo]{Lemma}

\newtheorem{prop}[theo]{Proposition}
\newtheorem{proposition}[theo]{Proposition}

\theoremstyle{definition}

\newtheorem{definition}[theo]{Definition}
\theoremstyle{remark}

\usepackage{lmodern}
\usepackage{caption}
\usepackage{color}
\usepackage{dsfont}

\newcommand\tempskipped[1]{}

\renewcommand\Re{\operatorname{Re}}
\renewcommand\Im{\operatorname{Im}}

\usepackage{tikz}
\usetikzlibrary{calc}
\usepackage{graphicx}

\newcommand\cQ{\mathcal{Q}}
\newcommand\cE{\mathcal{E}}
\newcommand\cS{\mathcal{S}}
\newcommand\cX{\mathcal{X}}
\newcommand\cT{\mathcal{T}}

\newcommand\dm{\diamond}
\newcommand\Dm{\diamondsuit}

\newcommand\Unif{{\mbox{\textsc{Unif(}$\delta$\textsc{)}}}}
\newcommand\Uniffff{{\mbox{\textsc{Unif(}$1$\textsc{)}}}}

\newcommand\Uniff{{\mbox{\textsc{Unif(}$\delta,r_0,\theta_0$\textsc{)}}}}

\newcommand\vcirc[1]{v^\circ_1,\ldots,v^\circ_{#1}}
\newcommand\svcirc[1]{\sigma_{v^\circ_1}\ldots \sigma_{v^\circ_{#1}}}
\newcommand\vbullet[1]{v^\bullet_1,\ldots,v^\bullet_{#1}}
\newcommand\muvbullet[1]{\mu_{v^\bullet_1}\ldots\mu_{v^\bullet_{#1}}}

\bibliography{Completebibliography.bib}

\begin{document}

\title{Universality of Ising spin correlations on critical doubly-periodic graphs}
\author[Rémy Mahfouf]{Rémy Mahfouf$^\mathrm{A}$}

\thanks{\textsc{${}^\mathrm{A}$ Université de Genève.}}

\thanks{\emph{E-mail:} \texttt{remy.mahfouf@unige.ch}}

\maketitle

\begin{abstract}
We establish conformal invariance of Ising spin correlations on critical doubly periodic graphs, showing that their scaling limit coincides with that of the critical square lattice, as originally proved in \cite{CHI}. To overcome the absence of integrability and quantitative full plane constructions in the periodic setting, we combine discrete analytic tools with random cluster methods. This result completes the universality picture for periodic lattices, whose criticality condition was identified in \cite{cimasoni-duminil} and whose conformal structure and interface convergence were obtained in \cite{Che20}.
\end{abstract}

\section{Introduction}
\label{sec:introduction}
\subsection{General context} The Ising model, introduced by Ising and Lenz \cite{ising1925beitrag}, remains one of the most studied models in probability and statistical mechanics, especially in the planar nearest-neighbour case without external field (see, e.g., the monographs \cite{friedli-velenik-book,mccoy-wu-book,palmer2007planar} and references therein). We adopt a convention dual to the standard setup, assigning $\pm1$ spins to the set $G^\circ$ of faces of a planar graph $G$. When $G$ is finite and connected, one specifies to each edge $e \in E(G)$, which separates two faces $v^\circ_{\pm}(e) \in G^\circ$, a positive coupling constant $J_e$. This ferromagnetic model favours configurations in which neighbouring spins align. For a fixed inverse temperature $\beta > 0$, one defines a probabilistic model on spin configurations $\sigma \in \{\pm1\}^{G^\circ}$, with partition function given by
\begin{equation}
\label{eq:intro-Zcirc}
\mathcal{Z}(G)\ :=\ \sum_{\sigma:G^\circ\to\{\pm 1\}}\exp\Bigl[\beta\sum_{e\in E(G)}J_e\sigma_{v^\circ_-(e)}\sigma_{v^\circ_+(e)}\Bigr].
\end{equation}
In the homogeneous Ising model on the square lattice, a phase transition occurs between the paramagnetic and ferromagnetic phases as the inverse temperature varies. At the critical temperature- the \emph{Curie point} observed in early magnetic experiments \cite{curie1894possibilite} - the model exhibits remarkable features, revealing deep links between statistical mechanics, complex analysis, and algebra. Following Smirnov’s introduction of discrete fermionic observables on the square lattice \cite{Smi-ICM06,Smirnov_Ising}, a series of works established the conformal invariance (or covariance) of the critical model \cite{Smirnov_Ising,CHI,CheIzy13,HS-energy-Ising,Izy-phd}, as predicted by Conformal Field Theory (see, e.g., \cite{Zam,Zam2}). In essence, these results show that, on bounded domains, the critical Ising model admits a continuous scaling limit as the mesh size tends to zero, and that limits on conformally equivalent domains are related by the corresponding conformal map. This programme was later extended to \emph{critical isoradial graphs} for correlation functions \cite{CIM-universality}, while the convergence of interfaces \cite{BenHon19,SmirnovEtAlu} follows from essentially the same arguments as in the square-lattice case. Further studies on isoradial graphs investigated \emph{near-critical models} with Baxter’s Z-invariant weights \cite{BdTR1,BdTR2,Baxter-book}, approaching criticality \cite{park2018massive,park-iso,CIM-universality} and confirming long-standing predictions linking the model to the massive Dirac equation \cite{mccoy1977painleve,sato1979holonomic}.

Extending this framework beyond the isoradial case introduces significant challenges, due to the lack of a natural discrete conformal structure for analysing Ising fermions. The notion of $s$-embeddings, introduced in \cite{Ch-ICM18,Che20}, provides such a structure, with critical doubly-periodic graphs as a fundamental new example. The criticality condition for these graphs, obtained by Cimasoni and Duminil-Copin \cite{cimasoni-duminil} via a connection to the dimer model, is expressed algebraically in terms of edge weights on the fundamental domain. In \cite{Che20}, it was shown that every critical Ising model on a doubly-periodic graph admits (up to rotation, scaling, and complex conjugation) a \emph{unique} associated $s$-embedding. This enabled the construction of a discrete complex analysis framework, allowing Chelkak in that same article to prove the convergence of FK-Ising interfaces to SLE$(16/3)$, as in the square-lattice case. Studying the critical Ising model on doubly-periodic graphs falls within Kadanoff's \emph{universality} framework, asserting that microscopic details of the model should have little impact on its large-scale behaviour, which depends only on a few global parameters. However, while the convergence of FK interfaces is now established for such graphs, the behaviour of correlation functions remains unresolved. The powerful integrable techniques used on isoradial graphs, which rely on constructing full-plane fermionic observables with prescribed singularities and asymptotics, do not extend to critical doubly-periodic lattices. In this work, we establish the universality of Ising spin correlations on critical doubly-periodic graphs using soft discrete complex analysis methods combined with random-cluster techniques.

\subsection{The critical Ising model on doubly-periodic graphs}

When working on isoradial lattices, the notion of a critical model often appears naturally by choosing combinatorial weights that admit a \emph{geometric} interpretation (see, e.g., \cite{ChSmi2} for the FK-Ising model, \cite{grimmett2013universality,grimmett2014bond} for percolation, \cite{BdTR1,BdTR2,boutillier2011critical} for dimers, and \cite{ChSmi1} for random walks and discrete harmonic functions). These critical weights turn out to be invariant under the star-triangle transformation of the associated model. On isoradial grids there exist two notions of star-triangle deformation. On the one hand, there is a combinatorial star-triangle transformation of the model's weights, which performed in a suitable way, does not change the probability measure of events far from the local modification. On the other hand, there exists an exchange of \emph{train tracks} in isoradial lattices, corresponding to some purely geometric moves on isoradial lattices. The geometric formulae used to define critical isoradial weights make those two frameworks match. It is known that general doubly-periodic lattices equipped with the Ising model cannot admit an isoradial embedding \cite{KenyonSchlenker}. Therefore, trying to find a geometric criticality condition on the weights starting from a given isoradial picture does not work for the generic periodic case. Nevertheless, a way forward was found in two different articles using the exact bosonisation \cite{Dubedat-bos} of the Ising model as a dimer model. When the doubly-periodic graph has a $\mathbb{Z}^2$ combinatorial structure, a connection to the dimer context was first established by Li \cite{li2012critical} and later generalised in \cite{cimasoni-duminil} to all doubly periodic graphs. To be more precise, let us introduce some additional notation. Consider a weighted planar graph $(G,x)$ embedded in the plane up to homeomorphism preserving the cyclic ordering of edges. Denote its vertices by $G^{\bullet}$ and its faces by $G^{\circ}$. The bipartite graph $\Lambda(G) := G^{\bullet} \cup G^{\circ}$ has edges connecting each vertex to the faces it belongs to. Each quad $z_e = (v^{\bullet}_0 v^{\circ}_0 v^{\bullet}_1 v^{\circ}_1)$ of $\Lambda(G)$ corresponds to an edge $e$ of $G$, linking $v^{\bullet}_0$ and $v^{\bullet}_1$ and separating the faces $v^{\circ}_0$ and $v^{\circ}_1$. We denote by $e^\star$ the dual edge linking $v^{\circ}_0$ and $v^{\circ}_1$ in the graph $G^{\circ}$, with $(e^\star)^\star = e$. Under this identification, one can parametrise the coupling constant $x(e)$ using the abstract angle as
\begin{equation}
\label{eq:x=tan-theta}
\theta_{z(e)}\ :=\ 2\arctan x(e)\ \in\ (0,\tfrac{\pi}{2}), 
\quad x(e):=\exp[-2\beta J_e].
\end{equation}
It turns out that in the periodic case, the criticality condition can be written as a purely algebraic statement on the weights $x(e)$ over the fundamental domain. Let $(\mathcal{G},x)$ be a planar, non-degenerate, locally finite, doubly-periodic weighted graph. Denote by $\Gamma$ its fundamental domain, naturally embedded in the torus (with no specification on that torus embedding). Let $\mathcal{E}(\Gamma)$ be the set of even subgraphs of $\Gamma$, and $\mathcal{E}_0(\Gamma)$ the subset of even subgraphs of $\Gamma$ that wind an \emph{even} number of times around each of the two directions of the torus. Finally, denote $\mathcal{E}_1(\Gamma):=\mathcal{E}(\Gamma) \setminus \mathcal{E}_0(\Gamma)$. With this terminology, the criticality condition on doubly-periodic graphs is stated in the following theorem.
\begin{theo}[Cimasoni-- Duminil-Copin\cite{cimasoni-duminil}]\label{thm:cimasoni-duminil}
The Ising model $(\mathcal{G},x)$ is critical if and only if
\begin{equation}
\sum\limits_{\gamma \in \mathcal{E}_0(\Gamma)}x(\gamma)
-\sum\limits_{\gamma \in \mathcal{E}_1(\Gamma)}x(\gamma)=0,
\end{equation}
where $x(\gamma):=\prod_{e\in \gamma}x_e$. Moreover, the sign of the LHS of the above equation fixes the phase of the model (paramagnetic or ferromagnetic).
\end{theo}

This criticality condition is derived from the bosonisation of the Ising model and from the fact that the associated Kac-Ward determinant (evaluated at a pair of non-vanishing complex numbers) vanishes at $(1,1)$. As already mentioned, this condition is purely algebraic and \emph{does not} specify a natural doubly-periodic embedding into the plane, which would allow one to study a natural continuous limit of the associated random geometry. This question was settled by Chelkak in \cite[Lemma 2.3]{Che20} as follows.
\begin{lemma}[Lemma 2.3 in \cite{Che20}]
Let $(G,x)$ be a doubly-periodic graph carrying the critical Ising model in the sense of Theorem \ref{thm:cimasoni-duminil}. Then there exists (up to translation, rotation, scaling and complex conjugation) a unique doubly-periodic $s$-embedding $\mathcal{S}$ attached to $(G,x)$.
\end{lemma}

Note that this criticality condition can also be seen for the so-called FK-Ising model, whose definition is recalled below. Using the classical Kramers-Wannier duality, set the dual weight 
\begin{equation}\label{eq:Kramers-Wannier}
(x_{e})^\star:= \frac{1-x_e}{1+x_e}.
\end{equation}
When $G$ is a finite planar graph, the model with \emph{wired} boundary conditions can be embedded into the sphere, where a distinguished face $v_{\mathrm{out}}^{\circ}$ which interacts with \emph{all} boundary faces and carries a single fixed spin. The FK-Ising model on $G^\circ$ can then be interpreted as a probability measure on even subgraphs, such that for any subgraph $C$ of $G^\circ$ we have
\begin{equation}
	\mathbb{P}^{G^\circ}_{\mathrm{FK}}(C):=\frac{1}{Z_{\mathrm{FK}}(G^{\circ},(x_e)_{e\in G})}\,
	2^{\#\textrm{clusters}(C)}
	\prod_{e^\star\in C}(x_{e^\star})^\star
	\prod_{e^\star\not \in C}\bigl(1-(x_{e^\star})^\star\bigr),
\end{equation}
where $e^\star$ denotes the dual edge of $G$ linking the vertices $v^{\pm}_{e^\star} \in G^{\circ}$, $\#\mathrm{clusters}(C)$ is the number of clusters in the subgraph $C$, and $Z_{\mathrm{FK}}(G^{\circ}, (x_e)_{e \in G})$ is a normalization constant. It is standard (see, e.g., \cite{duminil-parafermions}) to pass to the infinite-volume limit, thereby defining a full-plane FK-Ising measure. In this paper, we only consider graphs satisfying the strong box-crossing property of Theorem \ref{thm:RSW-FK}, ensuring that the infinite-volume limit is unique and independent of the choice of boundary conditions on finite graphs. The (combinatorial) link between the Ising model with wired boundary conditions and the FK-Ising model with wired boundary conditions, known as the Edwards-Sokal coupling introduced in \cite{edwards1988generalization}, reads as follows:

\begin{itemize}
	\item \textbf{Ising model to FK-Ising model:} Start with a spin configuration $\sigma \in \{\pm 1 \}^{G^{\circ}}$ and, independently for each pair of aligned neighbouring faces $v_{\pm}^\circ\in G^{\circ}$, sample a Bernoulli random variable with parameter $(x_{e^\star})^\star$. The faces $v_{\pm}^\circ$ are connected in the random-cluster model if and only if the associated Bernoulli variable equals $1$. This constructs a random graph in $G^{\circ}$.
	\item \textbf{FK-Ising model to Ising model:} For each cluster $C$ in $G^\circ$, sample (independently from other clusters) a fair $\pm 1$ random variable and assign the resulting spin to \emph{all} vertices attached to $C$. 
\end{itemize}

For the rest of the article, fix some critical doubly-periodic graph and let $ \cS$ be a canonical doubly-periodic $s$-embedding associated with it. For simplicity, we assume that all the edges of $\cS$ have length comparable to $1$ and denote by $\cS^{\delta}:= \delta \cdot \cS $ its scaled version. Here \emph{there is no natural notion of scale} $\delta$, contrary to the isoradial case where primal to dual edges in $\Lambda(G)$ all have a common length. In the periodic context, on $\cS^\delta $ all the edges of $\Lambda(G)$ have length comparable to $\delta$, up to some uniform multiplicative constant, while all the angles are uniformly bounded away from $0$ and $\pi$. All the estimates in the following theorems will implicitly depend on those factors governing the geometry of $\cS$. This setup for $\cS^\delta $ can be summarised as the assumption \Unif\, introduced in  \cite{Che20} and recalled in Section \ref{sub:notation}.  

Following the work of \cite{Che20}, it is now understood that the criticality condition has much greater implications than merely a phase transition from paramagnetic to ferromagnetic: it can already be formulated as a manifestation of conformal invariance of the critical model. Before giving precise details about the convergence of FK-interfaces, let us state the first major result of  \cite{Che20} in the critical periodic context.  
Fix an annulus $\Box(\ell):=[-3\ell;3\ell]^2 \setminus [-\ell;\ell]^2 $ and set $\Box^\delta(\ell):=\Box(\ell) \cap \cS^\delta$ a discretization of the original lattice up to $10\delta$ (meaning here that the Hausdorff distance between the boundaries of $\Box(\ell)$ and $\Box^\delta(\ell):=\Box(\ell)$, seen as planar curves, is at most $10\delta$). In the following theorem, denote by $\mathbb{P}_{\textrm{FK}}^{\text{free}} $ the FK-Ising measure on $\Box^\delta(\ell)$ with \emph{free} boundary conditions, whose weights are inherited from those of $\cS^\delta$ via the Edwards-Sokal coupling. A open circuit of edges in $\Box^\delta(\ell)$ is a cluster of open edges in $\Box^\delta(\ell)$ surrounding its inner boundary. One has the following theorem.
\begin{theo}[Chelkak \cite{Che20}]\label{thm:RSW-FK}
In the above context, there exist positive constants $c_0,L_0>0$, depending only on the constants in \Unif\,, such that for any $\ell \geq L_0\cdot\delta $,
		\begin{equation}
			\mathbb{P}_{\mathrm{FK}}^{\mathrm{free}} \bigl[\mathrm{there\ exists\ an\ open\ circuit\ of\ edges\ in\ } \Box^\delta(\ell)\bigr] \geq c_0.
		\end{equation}
\end{theo}
This theorem shows that, as the critical square-lattice, crossing probabilities for FK clusters in large boxes remain bounded away from $0$ and $1$, in sharp contrast with the off-critical phase. It also implies the absence of an infinite cluster on critical doubly-periodic graphs, as originally proven in \cite{DumLis}. 

Finding the correct discrete conformal structure to this critical model even allowed going further and proving conformal invariance of the FK-interfaces. Consider a simply connected domain $(\Omega^\delta,a^\delta,b^\delta) \subset \cS^\delta$ and define the Ising model on $(\Omega^\delta,a^\delta,b^\delta)$ with Dobrushin boundary conditions, \emph{wired} along the arc $(a^\delta b^\delta)^\circ $ and \emph{free} along the arc $(b^\delta a^\delta)^\bullet$. In this setup $a^\delta,b^\delta$ are \emph{corners} of the $s$-embedding $\cS^\delta$, linking $(a^\delta b^\delta)^\circ $ to $(b^\delta a^\delta)^\bullet$. Set $\gamma^\delta $ the discrete interface that separates primal and dual clusters in $(\Omega^\delta,a^\delta,b^\delta)$ and connects $a^\delta$ to $b^\delta$. In that context, conformal invariance of the FK-interfaces reads as follows.
\begin{theo}[Chelkak \cite{Che20}]\label{thm:SLE-periodic}
 If the sequence of domains $(\Omega^\delta,a^\delta,b^\delta)_{\delta>0}$ converges (in the Carathéodory sense) to a simply connected domain $(\Omega,a,b)$ with two marked boundary points $a,b\in \partial \Omega$, then the sequence of interfaces $(\gamma^\delta)_{\delta >0} $ converges in law to the continuous process $\mathrm{SLE}(16/3,a,b)$ in $\Omega$.
\end{theo}

\subsection{Main results}
We establish conformal invariance of Ising spin correlations on all critical doubly-periodic graphs. There are generally two kinds of global observables that come into play when proving conformal invariance statements. One can first study interfaces between Ising or FK-Ising clusters, aiming to show their convergence to SLE or CLE processes. In a second step, one seeks to establish the convergence of correlation functions, either for spins in generic positions inside a domain or for the energy density variables encoding products of nearby spins. While these results are by now well understood for critical and near-critical isoradial grids, their generalisation to doubly-periodic grids were still missing.

The energy random variable \emph{can} be expressed in terms of fermions, allowing one to use a clever integration trick from \cite{MahPar25a} to recover conformal covariance of the second term in the expansion of the energy density. Hence, the only remaining challenge concerns the spin correlations, which, even on the square lattice, require a much more delicate analysis, as they can be expressed through \emph{ratios} of fermions rather than their products. For a generic periodic lattice, we say that two faces $f^\delta \sim f'^\delta$ are of the \emph{same type} if they have the same image when projected onto the fundamental domain $\Gamma^\delta$ of $\cS^\delta$. There are finitely many such types, denoted $\cT := \{ \frak{f}_1, \dots, \frak{f}_m \}$. Fix a simply connected domain $\Omega$ with smooth boundary, approximated (in the Hausdorff sense) by $\Omega^\delta \subset \cS^\delta$, and faces of same type $a^\delta,b^\delta,c^\delta,d^\delta \in \Omega^\delta$ approximating respectively distinct points $a,b,c,d$ in the interior of $\Omega$. We settle the question of Conformal invariance of spin correlations in critical doubly-periodic grids in the following theorem.
\begin{theo}\label{thm:Correlation-function}
	 In the previous framework, for the Ising model on $\Omega^\delta$ with \emph{wired} boundary conditions, one has
	\begin{equation}
		\frac{\mathbb{E}^{\mathrm{(w)}}_{\Omega^\delta}[\sigma_{a^\delta}\sigma_{b^\delta}]}{\mathbb{E}^{\mathrm{(w)}}_{\Omega^\delta}[\sigma_{c^\delta}\sigma_{d^\delta}]} \underset{\delta \to 0}{\longrightarrow} \frac{\langle \sigma_a \sigma_b \rangle_{\Omega}^{(\mathrm{w})}}{\langle \sigma_c \sigma_d \rangle_{\Omega}^{(\mathrm{w})}},
	\end{equation}
where the correlation function $\langle \cdots  \rangle_{\Omega}^{(\mathrm{w})}$ is recalled in Section~\ref{sub:correlation} and coincides with that of the square lattice. The convergence is uniform on compacts of $\Omega$ and with respect to the distance separating the points $a,b,c,d$. Moreover, there exist some lattice-dependent constants $C_{(\frak{f}_{a^\delta})},C_{(\frak{f}_{b^\delta})}$, depending only on the types of $a^\delta$ and $b^\delta$, and a global scaling factor $\rho(\delta)=\delta^{-\frac{1}{8}+o(1)} $ such that
\begin{equation}
	\rho(\delta)^{-2}\cdot\mathbb{E}^{\mathrm{(w)}}_{\Omega^\delta}[\sigma_{a^\delta}\sigma_{b^\delta}] \underset{\delta \to 0}{\longrightarrow}  C_{(\frak{f}_{a^\delta})}C_{(\frak{f}_{b^\delta})} \cdot \langle \sigma_a \sigma_b \rangle_{\Omega}^{(\mathrm{w})}.
\end{equation}
\end{theo}
For clarity, we restrict the theorem to the case of two spins, though our approach extends directly to any number $n>2$ of spins. The continuous correlation function is \emph{conformally covariant}, with conformal weight $\frac{1}{8}$. The scaling factors $\rho(\delta)^{-2}$ were obtained by explicit computations on the square lattice \cite{McWu-Ising} (see also \cite{CHM-zig-zag-Ising} for a modern derivation) and later extended to isoradial lattices by a gluing argument. However, no derivation of these scaling factors appears to exist in the general setting of critical doubly-periodic graphs. The SLE/CLE-based approaches only yield the asymptotic order $r(\delta)=\delta^{-1/8+o(1)}$, possibly with logarithmic corrections \cite{wu2018polychromatic,garban2020convergence}. For percolation, \cite{du2024sharp} obtained the result without logarithmic correction, relying on quantitative convergence of interfaces established in \cite{binder2024power}. Finding a closed formula for $\rho(\delta)$ amounts to determining the sharp asymptotics of the full-plane two-point function on periodic lattices. Existing methods on regular lattices, based on determinant computations and orthogonal polynomials \cite{CHM-zig-zag-Ising}, does not seem straightforward.

As it will appear in the proof, we compensate for the lack of integrability of full-plane discrete fermions by relying on the convergence of FK-Ising clusters to the CLE$(16/3)$ loop soup. This statement was established in \cite{KemSmi2}\footnote{This article proved the convergence in the natural topology for loop collections (see, e.g., \cite[Section 2.7]{BenHon19}). We use here this convergence to establish discrete/continuum analogs of connection probabilities for \emph{very simple connection events}. One may also wonder about convergence in the Schramm–Smirnov topology \cite{schramm2011scaling}, which is more suitable for studying connectivity properties. This convergence is widely believed to hold, although precise references may still currently be lacking.}. One can repeat the proof of the square lattice case, as it amounts, among other things, to combine the convergence of the martingale observable proved in \cite{Che20} with the \emph{Super Strong Box Crossing Property} that we prove in Theorem \ref{thm:SuperStrongRSW}. 

Consider a discrete topological rectangle $\mathcal{D}:=(a^\delta b^\delta c^\delta d^\delta)\subset\cS^\delta$, and denote by $\ell_{\mathcal{D}}[(a^\delta b^\delta),(c^\delta d^\delta)]$ its \emph{extremal length} as a discrete domain in the plane $\mathbb{C}$. There exists a unique $\ell=\ell_{\mathcal{D}}[(a^\delta b^\delta),(c^\delta d^\delta)]$ such that $\mathcal{D}$ is conformally mapped to $[0,1]\times[0,\ell]$, with $a^\delta,b^\delta,c^\delta,d^\delta$ mapped respectively to the corners of this rectangle in counter-clockwise order, starting from the lower-left corner. The next theorem states that crossing probabilities with free boundary conditions can be uniformly bounded in terms of the extremal length of the discrete domain.
\begin{theo}\label{thm:SuperStrongRSW}
	There exist constants $\eta,M>0$, depending only on the constants in \Unif\,, such that:
	\begin{itemize}
		\item If $\ell_{\mathcal{D}}[(ab),(cd)]\leq M$ then $ \mathbb{P}_{\mathcal{D}}^{\mathrm{free}}[(ab)\leftrightarrow(cd)]\geq \eta$.
		\item If $\ell_{\mathcal{D}}[(ab),(cd)]\geq M$ then $ \mathbb{P}_{\mathcal{D}}^{\mathrm{free}}[(ab)\leftrightarrow(cd)]\leq 1- \eta$.
	\end{itemize}
\end{theo}

This statement is \emph{not} known beyond the isoradial setting, where it was established in \cite{CDH} using sharp random-walk estimates on isoradial grids developed in \cite{Chetoolbox}. An alternative proof based on a renormalisation argument on the square lattice is presented in \cite{DMT21}, but this approach does not carry over directly to the periodic case. In particular, a crucial step in that argument relies on the use of primal and dual one-arm exponents in the half-plane, which cannot be extended beyond the square lattice without additional work, due to the lack of self-duality, that we bypass here using new discrete analytic tools.

\subsection{Novelties of the paper and adopted strategy}

Beyond the RSW crossing estimates of Theorem~\ref{thm:RSW-FK}, establishing the convergence of FK interfaces via Smirnov's original approach~\cite{Smirnov-conformal} requires addressing a challenging discrete to continuous boundary analysis problem for the scaling limit of Ising fermions. This problem was resolved for periodic grids by Chelkak~\cite{Che20} using a novel method that fundamentally differs from Smirnov's original proof. The study of correlation functions in bounded domains is substantially more involved. Although the boundary analysis of fermions associated with energy density and spin correlations resembles that of the FK observable, these fermions possess \emph{discrete singularities} in the bulk. On isoradial lattices, one way to analyse them goes by constructing full-plane fermions with prescribed discrete singularities (see \cite{HonglerSmirnov,HonPHD,CHI,CHI-mixed,CIM-universality,park2018massive}). These constructions can be reformulated in terms of \emph{discrete exponentials}, a one-parameter family (indexed by $\lambda\in\mathbb{C}$) of $s$-holomorphic functions that, when $\lambda$-integrated along suitable contours, yield explicit analytic expressions for full-plane fermions. However, no simple geometric formula for these discrete exponentials is known on periodic lattices corresponding to $s$-embeddings, causing the isoradial proof strategy to break down.

Recently, \cite{MahPar25a} showed that the convergence of the energy density follows from a simple integration trick yielding explicit local scaling factors, characterized by a non-trivial residue of a full-plane Green function. These factors are now understood even on irregular grids and for non-trivial limiting conformal structures. In contrast, the convergence of spin correlations, developed in \cite{CHI,CHI-mixed,CIM-universality}, relies on a substantially heavier construction. The central step consists in controlling the \emph{logarithmic derivative} of Ising correlation functions, that is, expanding
\[
\log\frac{\mathbb{E}_{\Omega^\delta}[\sigma_{a^{\delta}}\sigma_{b^\delta}]}{\mathbb{E}_{\Omega^\delta}[\sigma_{a'^\delta}\sigma_{b^\delta}]}=\mathcal{A}_{(a,b)}\cdot \delta + o(\delta),
\]
where the spins $a^\delta$ and $a'^\delta$ are one step apart, while $b^\delta$ remains fixed. The coefficient $\mathcal{A}_{(a,b)}$ is related to a Riemann--Hilbert boundary value problem in complex analysis, recalled in Section~\ref{sub:correlation}. Intuitively, one studies how the ratio of correlation functions changes when the spin under consideration is shifted by one step, and relates this discrete variation to a continuous function. Once this one-step change is understood, the result can be integrated along arbitrary discrete paths to establish convergence of correlation ratios. In \cite{CHI}, this step required a delicate analysis relying on symmetry arguments and explicit constructions, later simplified in \cite{CIM-universality} through another integration trick. However, that approach requires precise \emph{two-term asymptotics} of full-plane discrete fermions, which appears intractable on periodic lattices without an explicit discrete exponential representation.

In this paper, we take a different approach. We begin with new and fairly simple observations in discrete complex analysis on the boundary behavior of FK observables near flat domains, from which we deduce that the one arm exponent in the half plane is \emph{exactly} $1/2$. This replaces the boundary Harnack principle of \cite[Section~5]{ChSmi2}, which holds on isoradial lattices and relies on sub and superharmonicity arguments for primitives of squared fermions together with sharp random walk estimates near straight boundary arcs, tools that do not apply in the periodic setting. We then apply the bootstrap method of \cite{DMT21} to prove Theorem~\ref{thm:SuperStrongRSW}, supplemented by surgery techniques on $s$-embeddings in the spirit of \cite{Mah23,MahPHD}, which attach layers of boundary kites to a piece of a periodic lattice. This construction provides effective control of the boundary argument of fermionic observables, in close analogy with the square lattice case. As a byproduct, we obtain a new way to control boundary conditions for discrete fermionic observables in smooth domains, applicable to periodic graphs and extendable beyond this setting. In particular, it can be used to recover convergence of FK observables in smooth domains approximated in the Hausdorff sense in $\cS^\delta$ (see for instance \cite[Chapter~6]{MahPHD}).

All together, this provides the missing ingredients to deduce the CLE convergence of \cite{KemSmi2}. At this point, we take a route that hasn't been used before in computing correlations for the Ising model. Using I.I.C.\ arguments, we express connections probabilities in FK cluster via connection probabilities between small but macroscopic loops, which converge to CLE connections probabilities in the limit $\delta \to 0 $. Note that those crossing probabilities between small annuli are a priori hard to compute. Instead of diving into a difficult analysis, we make the \emph{same} reasoning on the square lattice, where convergence of correlation functions is known by \cite{CHI}, which allows to conclude without making \emph{any} CLE computation. Finally, the passage to the full-plane renormalisation follows the approach of the original version of \cite{CHI}, where scaling factors were expressed through full-plane two-point correlations, without assuming the explicit closed form $\delta^{-1/8}$ per spin derived from the McCoy-Wu computation on the square lattice \cite{McWu-Ising}.
\medskip
\section*{Acknowledgements}
I would like to thank Hugo Duminil-Copin for encouraging me to write down this research and for suggesting to bypass CLE computations using known results on the square lattice, as well as Emile Avérous and Tiancheng He for many enlightening explanations on the Incipient Infinite Cluster. I would like to thank Christophe Garban for his encouragements and Paul Cahen and Stanislav Smirnov for useful references on CLE convergence. Early stages of this research were initiated while the author was a PhD student at ENS Paris. I would like to thank the institution for its support, as well as Dmitry Chelkak for discussions on the discrete complex analysis. This research is funded by Swiss National Science Foundation and the NCCR SwissMAP.

\section{Notations and crash intro into the s-embedding formalism}\label{sec:notations}
\setcounter{equation}{0}

We briefly recall the general framework of $s$-embeddings as introduced in \cite[Section 3]{Che20}, together with the associated regularity theory for the so-called $s$-holomorphic functions, both of which arise from a complexification of the classical Kadanoff-Ceva formalism. Our notation is taken directly from \cite{Che20, Mah23, MahPar25a} and is consistent with \cite[Section~3]{CCK}, \cite{Ch-ICM18}, and \cite{CLR1,CLR2}. Since we do not include proofs here, we refer the reader to \cite[Section 2]{Che20} for more details. Chelkak’s main idea was to define a class of embeddings associated with a given weighted abstract graph, where the weights have a natural geometric interpretation, allowing the use of discrete complex analysis methods.

\begin{figure}\begin{minipage}{0.49\textwidth}
\includegraphics[clip, width=0.8\textwidth]{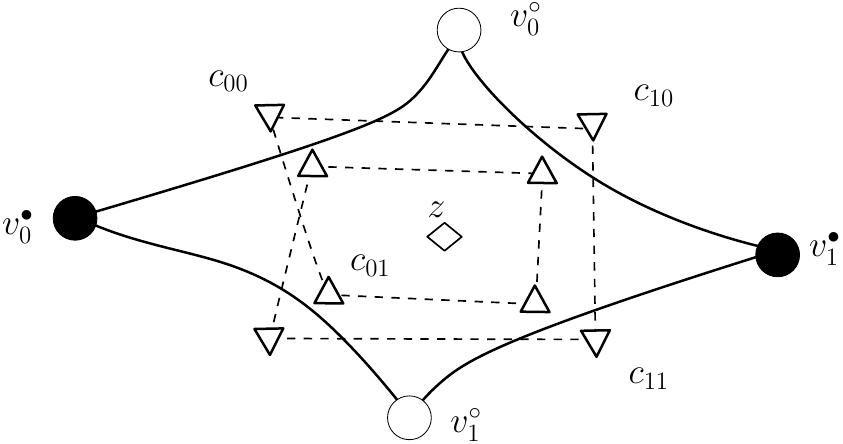}
\end{minipage}\begin{minipage}{0.49\textwidth}
\includegraphics[clip, width=0.7\textwidth]{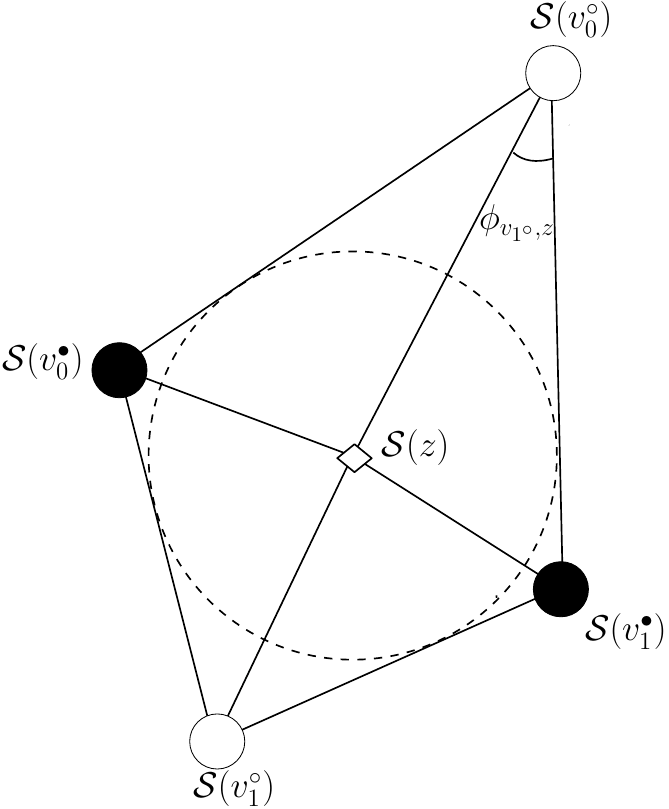}
\end{minipage}
\caption{(Left) Notation for a quad $z \in \diamondsuit(G)$ with an arbitrary planar embedding. Vertices of the primal graph $G^\bullet$ are indicated as black dots, and vertices of the dual graph $G^\circ$, corresponding to faces of $G$, are shown as white dots. The so-called corners, corresponding to edges of the bipartite graph $\Lambda(G) = G^\bullet \cup G^\circ$, are represented as triangles. This illustration shows a part of the \emph{double cover} of the corner graph branching around $z$. Corners that are neighbours \emph{within this double cover} are connected by dashed lines. (Right) A section of the associated $s$-embedding containing the quad $\cS^{\diamondsuit}(z)$, tangent to a circle of radius $r_z$ centreed at $\cS(z)$. The Ising weight of the edge between $v_0^\bullet$ and $v_1^\bullet$ can be determined from the four angles $\phi_{v,z}$ associated with $\cS^{\diamondsuit}(z)$, using the formula in \eqref{eq:theta-from-S}.}
\label{fig:graph-notations}
\end{figure}

\subsection{Notation and Kadanoff-Ceva formalism}\label{sub:notation}

Let us fix a planar graph $G$, which may include multiple edges and vertices of degree 2, but excludes loops and vertices of degree 1, and whose combinatorial type corresponds either to the plane or to the sphere. We consider $G$ up to homeomorphisms that preserve the cyclic order of edges around each vertex. In the spherical setting, one face of $G$ is designated as the outer face. We denote the original graph by $G=G^\bullet$, with vertices labeled $v^\bullet\in G^\bullet$, and its dual by $G^\circ$, with vertices $v^\circ\in G^\circ$ corresponding to the faces of $G$. The faces of the graph $\Lambda(G):= G^\circ \cup G^\bullet$, which naturally forms a bipartite graph, are in one-to-one correspondence with the edges of $G$. Furthermore, we define $\Dm(G)$ as the dual graph of $\Lambda(G)$, whose vertices $z \in \Dm(G)$ correspond to the faces of $\Lambda(G)$. This graph is usually referred to as the quad graph. Finally, the medial graph $\Upsilon(G)$ of $\Lambda(G)$ is introduced, with vertices called \emph{corners} of $G$, each corresponding to an edge $(v^\bullet v^\circ)$ of $\Lambda(G)$.

To maintain full consistency with the Kadanoff-Ceva formalism, it is typically necessary to consider various double covers of $\Upsilon(G)$. For examples of such double covers, see \cite[Fig.~27]{Mercat-CMP} or \cite[Fig 3.A]{Che20}. In this paper, we denote by $\Upsilon^\times(G)$ the double cover that branches over \emph{all} faces of $\Upsilon(G)$, i.e., around every element $v^\bullet \in G^\bullet$, $v^\circ \in G^\circ$, and $z \in \Dm(G)$. When $G$ is finite, this definition is well-posed, since the quantity $\#(G^\bullet) + \#(G^\circ) + \#(\Dm(G))$ is always even. 
Given a set $\varpi = \{\vbullet{m}, \vcirc{n}\} \subset \Lambda(G)$, with both $m$ and $n$ even, we define $\Upsilon^\times_\varpi(G)$ as the double cover of $\Upsilon(G)$ that branches over all faces \emph{except} those in $\varpi$. 
Similarly, $\Upsilon_\varpi(G)$ denotes the double cover of $\Upsilon(G)$ branching \emph{only} over the faces in $\varpi$. A function defined on any of these double covers is called a \emph{spinor} if its values on the two distinct lifts of the same corner differ only by a sign, i.e.\, by a factor of $-1$.

In this work, we consider the Ising model defined on the faces of the graph $G$, including the outer face in the disc case, which corresponds to adopting \emph{wired} boundary conditions. In this statistical mechanics model, spins $\pm1$ are assigned randomly to the vertices of $G^\circ$, with probabilities determined by the partition function \eqref{eq:intro-Zcirc}. The associated low-temperature expansion \cite[Section 1.2]{CCK} provides a mapping from a spin configuration $\sigma : G^\circ \to \{\pm 1\}$ to a subset $C$ of edges in $G$ separating spins of opposite signs. This mapping is actually $2$-to-$1$, depending on the choice of spin at the outer face.

Let us fix an even number $n$ of vertices $\vcirc{n} \subset G^\circ$, and consider a subgraph $\gamma^\circ = \gamma_{[\vcirc{n}]} \subset G^\circ$ that has odd degree precisely at the vertices in $\vcirc{n}$ and even degree at all other vertices of $G^\circ$. Such a subgraph can be interpreted as a collection of paths on $G^\circ$ pairing up the vertices in $\vcirc{n}$. Define
\[
x_{[\vcirc{n}]}(e)\ :=\ (-1)^{e \cdot \gamma_{[\vcirc{n}]}}\, x(e), \quad e \in E(G),
\]
where $e \cdot \gamma = 0$ if the edge $e$ does not cross $\gamma$, and $e \cdot \gamma = 1$ otherwise. Then the spin correlation can be expressed as
\begin{equation}
\label{eq:Esigma}
\textstyle \mathbb E\big[\svcirc{n}\big]\ =\ {x_{[\vcirc{n}]}(\cE(G))}\big/{x(\cE(G))},
\end{equation}
where $x(C):=\prod_{e\in C}x(e)$, $x(\cE(G)):=\sum_{c\in\cE(G)}x(C)$, and similarly for product of the kind $x_{[\vcirc{n}]}$.

Similarly, if $m$ is even and $\vbullet{m} \subset G^\bullet$, one can select a subgraph $\gamma^\bullet = \gamma^{[\vbullet{m}]} \subset G^\bullet$ with even degree at all vertices except those in $\vbullet{m}$. Following the Kadanoff-Ceva formalism \cite{kadanoff-ceva}, one can flip the signs of the interaction constants along $\gamma^\bullet$, $J_e \mapsto -J_e$, which is equivalent to replacing $x(e)$ by $x(e)^{-1}$ for edges in $\gamma^\bullet$, creating an anti-ferromagnetic region along $\gamma^\bullet$. This defines the random variable
\[
\textstyle \muvbullet{m}\ :=\ \exp\big[-2\beta\sum_{e\in\gamma^{[\vbullet{m}]}}J_e\sigma_{v^\circ_-(e)}\sigma_{v^\circ_+(e)}\,\big]\,.
\]
The domain-wall representation then implies (see, e.g., \cite[Proposition 1.3]{CCK})
\begin{equation}
\label{eq:Emu}
\textstyle \mathbb E\big[\muvbullet{m}\big]\ =\ \frac{x(\cE^{[\vbullet{m}]}(G))}{x(\cE(G))},
\end{equation}
where $\cE^{[\vbullet{m}]}(G)$ denotes the set of subgraphs with even degree at all vertices except those in $\vbullet{m}$, which have odd degree. Averaging in \eqref{eq:Emu} removes the dependence on the choice of $\gamma^\bullet$.  

Crucially, one can combine \eqref{eq:Esigma} and \eqref{eq:Emu} to handle configurations where both spins and disorder variables are present simultaneously. In this case, one obtains (see, e.g., \cite[Proposition 3.3]{CCK})
\begin{equation}
\label{eq:Emusigma}
\textstyle \mathbb E\big[\muvbullet{m}\svcirc{n}\big]\ =\ x_{[\vcirc{n}]}(\cE^{[\vbullet{m}]}(G))\big/{x(\cE(G))},
\end{equation}
Here, the variable \( \muvbullet{m} \) is defined as before. A subtlety arises for these mixed correlations: the sign of the expression in \eqref{eq:Emusigma} now depends on whether the number of intersections between the paths \( \gamma^\circ \) and \( \gamma^\bullet \) is even or odd. There is no canonical way to resolve this sign issue while remaining within the Cartesian product structure \( (G^\bullet)^{\times m} \times (G^\circ)^{\times n} \). To handle this, one may fix an embedding \( \cS : \Lambda(G) \to \mathbb{C} \) of $G$ and consider the natural double cover of \( (G^\bullet)^{\times m} \times (G^\circ)^{\times n} \), branching precisely like the spinor  \( \left[ \prod_{p=1}^m \prod_{q=1}^n (\cS(v^\bullet_p) - \cS(v^\circ_q)) \right]^{1/2} \). Following the discussion in \cite[Section 2.2]{CHI-mixed}, the quantities in \eqref{eq:Emusigma} can then be viewed as spinors on this double cover. 
Moreover, when considering mixed correlations of the type \eqref{eq:Emusigma}, the usual Kramers-Wannier duality (see again \cite[Proposition 3.3]{CCK}) ensures that \( G^\bullet \) and \( G^\circ \) play symmetric roles.

Among all correlators of the form \eqref{eq:Emusigma}, a particularly useful case occurs when one disorder vertex \( v^\bullet(c) \in G^\bullet \) and one spin \( v^\circ(c) \in G^\circ \) are nearest neighbors in \( \Lambda(G) \), connected by an edge of $\Lambda(G)$ corresponding to a corner \( c \in \Upsilon(G) \). In this setting, one can formally define the \emph{fermion} at the corner $c$ as
\begin{equation}
\label{eq:KC-chi-def}
\chi_c := \mu_{v^\bullet(c)} \sigma_{v^\circ(c)}.
\end{equation}
Using \eqref{eq:Emusigma}, one can then construct the \emph{Kadanoff-Ceva fermion} combinatorially by
\begin{equation}
\label{eq:KC-fermions}
X_{\varpi}(c) := \mathbb{E} \big[\, \chi_c \, \mu_{v_1^\bullet} \cdots \mu_{v_{m-1}^\bullet} \, \sigma_{v_1^\circ} \cdots \sigma_{v_{n-1}^\circ} \, \big].
\end{equation}
Due to the previous considerations, the observable \( X_\varpi(c) \) is defined up to a sign; however, it becomes fully well-defined when considered on the double cover \( \Upsilon^\times_\varpi(G) \).  

For each quad \( z = (v_0^\bullet, v_0^\circ, v_1^\bullet, v_1^\circ) \) (ordered counterclockwise, see \cite[Figure 3.A]{Che20} or Figure \ref{fig:graph-notations}), the Kadanoff-Ceva observables satisfy simple local linear relations, with coefficients determined solely by the Ising coupling associated with the quad $z$. These propagation equations were first introduced in \cite{dotsenko1983critical}, \cite{perk1980quadratic}, and \cite[Section 4.3]{Mercat-CMP}. More concretely, let \( \theta_z \) denote the abstract angle parametrizing the edge in \( G^\bullet \) associated with $z$, as in \eqref{eq:x=tan-theta}. Then, for any triplet of corners \( c_{pq} = (v^\bullet_p v^\circ_q) \) whose lifts to \( \Upsilon^\times_\varpi(G) \) are neighbors, one has
\begin{equation}
\label{eq:3-terms}
X(c_{pq}) = X(c_{p,1-q}) \cos \theta_z + X(c_{1-p,q}) \sin \theta_z.
\end{equation}
It can be readily checked that any solution to \eqref{eq:3-terms} naturally defines a spinor on the double cover \( \Upsilon^\times_\varpi(G) \).

In the present work, the main observable that will play an explicit role in our arguments is the so-called FK-martingale observable. To set the stage, consider a discrete simply connected domain \( \Omega^\delta \subset \cS^\delta \) with two marked boundary corners \( a^\delta \) and \( b^\delta \). We study the Ising model on \( \Omega^\delta \) with \emph{wired} boundary conditions along the arc \( (a^\delta b^\delta)^\circ \subset \partial \Omega^\delta \) and \emph{free} boundary conditions along the complementary arc \( (b^\delta a^\delta)^\bullet \subset \partial \Omega^\delta \). The FK observable is then defined, for corners \( c \in \Upsilon^\times \cap \Omega^\delta \), by
\begin{equation}
X_{\Omega^\delta}^{\mathrm{FK}}(c) := \mathbb{E}_{\Omega^\delta} \big[ \chi_c \, \sigma_{(a^\delta b^\delta)^\circ} \, \mu_{(b^\delta a^\delta)^\bullet} \big].
\end{equation}
Observe that at the boundary corners connecting the wired and free arcs, one has \( |X_{\Omega^\delta}^{\mathrm{FK}}(a^\delta)| = |X_{\Omega^\delta}^{\mathrm{FK}}(b^\delta)| = 1 \), since the (multiplicative) contributions of spins and disorders cancel each other there.

\medskip

To complete this brief review of Kadanoff-Ceva observables, we recall the generalized Dirac spinor \( \eta_c \), which provides a particular solution to the linear equation \eqref{eq:3-terms} \emph{in the isoradial case}. Given an embedding \( \cS : \Lambda(G) \to \mathbb{C} \), we define, following \cite{ChSmi2}:
\begin{equation} \label{eq:def-eta}
\eta_c := \varsigma \cdot \exp\Big[-\frac{i}{2} \arg(\cS(v^\bullet(c)) - \cS(v^\circ(c)))\Big], \qquad \varsigma := e^{i \pi/4},
\end{equation}
where the prefactor \( \varsigma = e^{i\pi/4} \) is a convenient normalization. To remove the sign ambiguity in \eqref{eq:def-eta}, one can again work on the double cover \( \Upsilon^\times(G) \). In particular, the products \( \eta_c X_\varpi(c) : \Upsilon_\varpi(G) \to \mathbb{C} \) are well-defined on \( \Upsilon_\varpi(G) \), which only branches over the set \( \varpi \). For notational simplicity, we will continue to use \eqref{eq:def-eta} even in the context of embeddings \( \cS \) that are not isoradial.
\subsection{Definition of an s-embeddings}\label{sub:semb-definition}

We now describe the embedding procedure proposed by Chelkak in \cite[Section 6]{Ch-ICM18} and developed in full detail in \cite{Che20}. To start, we recall the notion of an s-embedding as in \cite[Definition 2.1]{Che20}, which is based on the Kadanoff-Ceva framework. The underlying philosophy is not to assign Ising weights to a predetermined tiling of the plane by tangential quadrilaterals; instead, one seeks an embedding that naturally accommodates the given Ising weights. The core idea is to use a solution to \eqref{eq:3-terms} to construct a concrete embedding of the weighted graph.
\begin{definition}\label{def:cS-def}
Let $(G,x)$ be a weighted planar graph with the combinatorial structure of the plane, and let $\cX : \Upsilon^\times(G) \to \mathbb{C}$ be a solution to the full system \eqref{eq:3-terms} around each quad. We say that $\cS = \cS_\cX : \Lambda(G) \to \mathbb{C}$ is an s-embedding associated with $\cX$ if, for each corner $c \in \Upsilon^\times(G)$,
\begin{equation}
\label{eq:cS-def}
\cS(v^\bullet(c)) - \cS(v^\circ(c)) = (\cX(c))^2.
\end{equation}
For $z \in \Dm(G)$, the quadrilateral $\cS^\dm(z) \subset \mathbb{C}$ is defined by the edges connecting the vertices $\cS(v_0^\bullet(z))$, $\cS(v_0^\circ(z))$, $\cS(v_1^\bullet(z))$, and $\cS(v_1^\circ(z))$. The s-embedding $\cS$ is called \emph{proper} if these quadrilaterals do not overlap, and \emph{non-degenerate} if no $\cS^\dm(z)$ collapses into a segment. No convexity assumption is required for the quads.
\end{definition}

Given a fixed solution $\cX$ to \eqref{eq:3-terms}, it is a non-trivial task to ensure that the resulting embedding $\cS_\cX$ is proper and non-degenerate. One can also define for $\cS$ the positions of quad centers $\cS(z)$, as in \cite[Equation (2.5)]{Che20}:
\begin{equation}\label{eq:cS(z)-def}
\begin{aligned}
\cS(v_p^\bullet(z)) - \cS(z) &:= \cX(c_{p0}) \cX(c_{p1}) \cos\theta_z,\\
\cS(v_q^\circ(z)) - \cS(z) &:= -\cX(c_{0q}) \cX(c_{1q}) \sin\theta_z,
\end{aligned}
\end{equation}
where $c_{p0}$ and $c_{p1}$ (respectively $c_{0q}$ and $c_{1q}$) are neighboring corners in $\Upsilon^\times(G)$. The propagation equation \eqref{eq:3-terms} ensures that both \eqref{eq:cS-def} and \eqref{eq:cS(z)-def} are consistent. Geometrically (see Figure \ref{fig:graph-notations}), the image $\cS^\dm(z) \subset \mathbb{C}$ of a combinatorial quad $z \in \diamondsuit(G)$ is a quadrilateral tangent to a circle centered at $\cS(z)$, with radius $r_z$ determined by the values of $\cX$ (see \cite[Equation (2.7)]{Che20}). Denoting by $\phi_{v,z}$ the half-angle at $\cS(v)$ in the quad, one can recover the Ising weight $\theta_z$ from the geometric angles via
\begin{equation}
\label{eq:theta-from-S}
\tan\theta_z = \Biggl(\frac{\sin\phi_{v_0^\bullet,z}\,\sin\phi_{v_1^\bullet,z}}{\sin\phi_{v_0^\circ,z}\,\sin\phi_{v_1^\circ,z}}\Biggr)^{1/2}.
\end{equation}
When working with $s$-embedding framework, the large-scale properties of the origami map coin the criticality of the model. The following definition is recalled from \cite[Definition 2.2]{Che20} (see also \cite{KLRR, CLR1} for a general construction in the dimer context).

\begin{definition}\label{def:cQ-def}
Given $\cS = \cS_\cX$, the \emph{origami} function, denoted  $\cQ = \cQ_\cX : \Lambda(G) \to \mathbb{R}$, is a real-valued function defined up to some global additive constant. Its increments between two neighboring vertices $v^{\bullet}(c)$ and $v^\circ(c)$ separated by the corner $c$ are given by
\begin{equation}
\label{eq:cQ-def}
\cQ(v^\bullet(c))-\cQ(v^\circ(c))\ :=\ |\cX(c)|^2\,=\,|\cS(v^\bullet(c))-\cS(v^\circ(c))|\,.
\end{equation}
\end{definition}
We will often use the notation $|\cX(c)|^2:=\delta_c$ corresponding to the length of the edge of $\Lambda(G)$ attached to the corner $c$. The alternate sum of edge-lengths in a tangential quadrilateral vanishes, which ensures that the definition of $\cQ$ is consistent. One can see the function $\cQ$ as a folding of the tangential  quadrilaterals along their diagonals (see e.g. \cite[Section 8.2]{CLR1}), which makes $\cQ$ a $1$-Lipschitz in the $\cS $ plane. In the periodic context, $\cQ$ is periodic. Moreover, the periodic grid $\cS^\delta $ satisfy some so-called \Unif\, property.
\begin{definition}[Assumption \Unif\,]
	We say that the $s$-embedding $\cS$ satisfies the assumption $\Unif\,=\Uniff\ $ for some parameters $\delta,r_0,\theta_0$ if all edge-lengths in $\cS$ are comparable to $\delta$, meaning that for any  neighbouring $v^{\bullet}\in G^\bullet$ and $v^{\circ}\in G^\circ$ one has
	\begin{equation}
		r_0^{-1}\cdot \delta \leq  |\cS(v^{\bullet})- \cS(v^{\circ})| \leq r_0\cdot \delta,
	\end{equation}
and all the geometric angles in the quads $\cS$ are bounded from below by $\theta_0$.
\end{definition}

\subsection{S-holomorphic functions and associated primitives}\label{sub:HF-def}
We now briefly recall the notion of \emph{s-holomorphic functions}, which has been extended to the context of $s$-embeddings in \cite{Che20}. This concept was first introduced by Smirnov \cite[Definition 3.1]{Smirnov_Ising} for the critical square lattice, and later generalized by Chelkak and Smirnov \cite[Definition~3.1]{ChSmi2} to isoradial graphs. In recent works, $s$-holomorphic functions play a central role in applying discrete complex analysis methods to the Ising model. We present the general definition following \cite[Definition 2.4]{Che20}, where $\textrm{Proj}[\cdot, \eta \mathbb{R}]$ denotes the standard projection onto the line spanned by $\eta \in \mathbb{C}$.

\begin{definition}\label{def:s-hol}
A function $F$ defined on a subset of $\Dm(G)$ is said to be $s$-holomorphic if, for any two adjacent quads $z, z' \in \Dm(G)$ sharing an edge $[\cS(v^\circ(c)); \cS(v^\bullet(c))]$ corresponding to a corner $c$, one has
\begin{equation}
\label{eq:s-hol}
\textrm{Proj}[F(z), \eta_c \mathbb{R}] = \textrm{Proj}[F(z'), \eta_c \mathbb{R}].
\end{equation}
\end{definition}

This definition provides a direct correspondence between real-valued solutions to \eqref{eq:3-terms} and complex-valued $s$-holomorphic functions, as first observed in \cite[Proposition 2.5]{Che20} and also discussed in \cite[Appendix]{CLR1}.

\begin{proposition}\label{prop:shol=3term}
Let $\cS = \cS_\cX$ be a proper $s$-embedding and let $F$ be $s$-holomorphic on $\Dm(G)$. For a quad $z \in \Dm(G)$ and a corner $c \in \Upsilon^\times(G)$ of $z$, one can define the spinor $X$ at $c$ by
\begin{align}
X(c)\ &:=\ |\cS(v^\bullet(c))-\cS(v^\circ(c))|^{\frac{1}{2}} \cdot \Re[\overline{\eta}_c F(z)] \notag\\
&= \ \Re[\overline{\varsigma} \cX(c) \cdot F(z)] \ = \ \overline{\varsigma} \cX(c) \cdot \mathrm{Proj}[F(z); \eta_c \mathbb{R}].
\label{eq:X-from-F}
\end{align}
The assignment $c \mapsto X(c)$ satisfies the three-term relations \eqref{eq:3-terms} around the quad $z$. Conversely, for any real-valued solution $X : \Upsilon^\times(G) \to \mathbb{R}$ of \eqref{eq:3-terms}, there exists a unique $s$-holomorphic function $F$ on $\Dm(G)$ such that \eqref{eq:X-from-F} holds.
Moreover, when $F$ and $X$ are related by \eqref{eq:X-from-F}, the value of $F(z)$ can be reconstructed from the values of $X$ at any two corners $c_{pq}(z) \in \Upsilon^\times(G)$ and the geometry of the embedding $\cS$, for instance via the formula \cite[Corollary 2.6]{Che20}:
\begin{equation} \label{eq:F-from-X}
F(z) = -i \varsigma \cdot \frac{\overline{\cX(c_{01}(z))} \, X(c_{10}(z)) - \overline{\cX(c_{10}(z))} \, X(c_{01}(z))}{\Im[\overline{\cX(c_{01}(z))} \, \cX(c_{10}(z))]}.
\end{equation}
\end{proposition}
Within the $s$-embedding framework, the large-scale behavior of $s$-holomorphic functions is controlled by their local equation and boundary data. At the discrete level, this can be analyzed through two distinct integration procedures. The first is a natural extension of the standard integration method for discrete holomorphic functions, adapted to account for the presence of the origami map. The second corresponds to a generalization of Smirnov’s primitive of the square of an $s$-holomorphic function. The first approach is studied in detail in \cite[Proposition 6.15]{CLR1} and is instrumental in deriving local regularity estimates and the limiting continuous local equation. The second approach, originally introduced by Smirnov in \cite{Smirnov_Ising} for the critical square lattice, provides a way to identify discrete Riemann-Hilbert boundary conditions in the Ising model. 

We start by defining the primitive $I_{\mathbb{C}}$. For an $s$-holomorphic function $F$ defined on $\diamondsuit(G)$, one can set (up to a global additive constant) \cite[Section 2.3]{Che20}:
\begin{equation}\label{eq:def-I_C}
I_{\mathbb{C}}[F]:= \int \big( \overline{\varsigma} F\, d\cS + \varsigma \overline{F} \, d\cQ \big).
\end{equation}
For a quad $z \in \diamondsuit(G)$ with vertices $v_{1,2}^{\bullet}, v_{1,2}^{\circ}$, one has, for $\star \in \{\bullet, \circ\}$,
\begin{equation}
I_{\mathbb{C}}[F](v_2^\star) - I_{\mathbb{C}}[F](v_1^\star) = \overline{\varsigma} F(z) \, [\cS(v_2^\star) - \cS(v_1^\star)] + \varsigma \overline{F(z)} \, [\cQ(v_2^\star) - \cQ(v_1^\star)].
\end{equation}
Thanks to extensions of the origami map to the full complex plane (see \cite{CLR1}, \cite[Section 2.3]{Che20}, or \cite[Section 2.4]{MahPar25a}), the definition \eqref{eq:def-I_C} can be consistently extended over all of $\mathbb{C}$.

Alternatively, one can define the \emph{primitive of the square} $H_X$ in a purely combinatorial manner, relying solely on the spinor $X$ satisfying the three-term relation \eqref{eq:3-terms}. This construction does not require an explicit embedding of the graph. Following \cite[Definition 2.8]{Che20}, one sets:

\begin{definition}\label{def:HX-def}
Let $X$ be a spinor on $\Upsilon^\times(G)$ solving \eqref{eq:3-terms}. Then $H_X$ is defined (up to a global additive constant) on $\Lambda(G) \cup \Dm(G)$ by
\begin{equation}
\label{eq:HX-def}
\begin{array}{rcll}
H_X(v_p^\bullet(z)) - H_X(z) &:=& X(c_{p0}(z)) X(c_{p1}(z)) \cos \theta_z, & p=0,1, \\[1mm]
H_X(v_q^\circ(z)) - H_X(z) &:=& - X(c_{0q}(z)) X(c_{1q}(z)) \sin \theta_z, & q=0,1, \\[1mm]
H_X(v_p^\bullet(z)) - H_X(v_q^\circ(z)) &:=& (X(c_{pq}(z)))^2,
\end{array}
\end{equation}
in analogy with \eqref{eq:cS-def} and \eqref{eq:cS(z)-def}.
\end{definition}

The consistency of this definition follows from \eqref{eq:3-terms}. When an $s$-embedding $\cS$ of $(G,x)$ is given, the correspondence between $X$ and the associated $s$-holomorphic function $F$ (as in Proposition \ref{prop:shol=3term}) allows one to interpret $H_X$ in terms of $F$. Concretely, one can define the function $H_F$ as in \cite[Equation (2.17)]{Che20}:
\begin{equation}
\label{eq:HF-def}
H_F := \int \Re\big(\overline{\varsigma}^2 F^2 d\cS + |F|^2 d\cQ \big) = \int \big(\Im(F^2 d\cS) + \Re(|F|^2 d\cQ) \big),
\end{equation}
defined on $\Lambda(G) \cup \Dm(G)$. By extending $\cQ$ piecewise-affinely over each face of the corresponding $t$-embedding $\cT = \cS$ (see \cite[Proposition 3.10]{CLR1}), $H_F$ can be consistently extended to the entire plane. This lemma confirms that the combinatorial definition \eqref{eq:HX-def} and the $s$-holomorphic-based definition \eqref{eq:HF-def} yield the \emph{same} function.

\begin{lemma}{\cite[Lemma 2.9]{Che20}} 
Let $F$ be defined on $\Dm(G)$ and let $X$ be a spinor on $\Upsilon^\times(G)$ related to $F$ via \eqref{eq:X-from-F}. Then, the functions $H_F$ and $H_X$ coincide up to an additive constant.
\end{lemma}

In the special case of an isoradial embedding $\cS$, the origami map $\cQ$ takes constant values on both $G^\bullet$ and $G^\circ$, as all edges of each quad $\cS^{\diamond}(z)$ have the same length. Consequently, $H_F$ reduces to the primitive of $\Im[F^2 d\cS]$, recovering the original construction from \cite[Section~3.3]{ChSmi2}. For the FK observable $X = X_{\Omega^\delta}^{\mathrm{FK}}$, the function $H_X$ satisfies the boundary conditions $H_X = 1$ along the free arc $(b^\delta a^\delta)^\bullet$ and $H_X = 0$ along the wired arc $(a^\delta b^\delta)^\circ$. In particular, the \emph{maximum principle} recalled in \cite[Proposition 2.11 and Corollary 2.12]{Che20} ensures that $H_X$ takes values in the interval $[0,1]$ throughout the domain $\Omega^\delta$.

Finally, we note an important property discussed in detail in \cite[Theorem 2.18 and Remark 2.12]{Che20}. In the periodic setting considered here, if a family of $s$-holomorphic functions $F^\delta$ is uniformly bounded on an open set $U$, then this family is precompact in the topology of uniform convergence on compact subsets as $\delta \to 0$. Moreover, these functions are $\beta(\kappa)$-H\"older continuous from scales comparable to $\delta$ onward. In particular, on compacts that are a fixed distance away from the boundary, the $s$-holomorphic function associated with the FK observable remains bounded and is $\beta(\kappa)$-H\"older at all scales. In particular, one can repeat the argument of \cite[Section 2.7]{Che20} to show that the scaling limit $f$ of a converging sequence $(F^\delta)_{\delta >0}$ is holomorphic on compacts.
\subsection{Basic properties of the FK random cluster model}\label{sub:basic-properties-FK}

In this section, we briefly recall some fundamental properties of the FK-Ising model, which will be essential for our subsequent analysis and for relating to the results of \cite{DMT21}. We denote the FK-Ising measure by $\phi$. Let $\leq$ denote the natural partial order on configurations $\{0,1\}^{E^\star}$: for $\omega,\omega'\in\{0,1\}^{E^\star}$, we write $\omega\le \omega'$ if $\omega(e^\star)\le \omega'(e^\star)$ for every $e^\star\in E^\star$. An event $A$ is called \emph{increasing} if whenever $\omega\in A$ and $\omega\le \omega'$, it follows that $\omega'\in A$. Intuitively, increasing events are preserved under the addition of open edges.  

Similarly, a boundary condition $\xi'$ is said to \emph{dominate} another boundary condition $\xi$ if every set of vertices wired together in $\xi$ is also wired in $\xi'$, which we denote $\xi'\ge \xi$. We summarize the key properties needed for our analysis below:

\begin{itemize}
    \item \textbf{Positive association:} For any two increasing events $A$ and $B$ under a boundary condition $\xi$,
    \begin{equation}
        \phi_{G}^{\xi}[A \cap B] \ge \phi_{G}^{\xi}[A] \cdot \phi_{G}^{\xi}[B].
    \end{equation}

    \item \textbf{Boundary monotonicity:} For an increasing event $A$ and boundary conditions $\xi'\ge \xi$,
    \begin{equation}
        \phi_{G}^{\xi'}[A] \ge \phi_{G}^{\xi}[A].
    \end{equation}

    \item \textbf{Domain Markov property:} Let $H$ be a subgraph of $G$, and $\xi$ a boundary condition on $G$. Then
    \begin{equation}
        \phi_{G}^{\xi}[\omega \textrm{ on } H \mid \omega \textrm{ on } G\backslash H] = \phi_{H}^{\zeta}[\omega \textrm{ on } H],
    \end{equation}
    where $\zeta$ is the induced boundary condition on $\partial H$, obtained by wiring together vertices that are connected in $G\backslash H$ through $\omega$, possibly also using connections prescribed by $\xi$.
    \item \textbf{Mixing property:} There exist a constant $c_{mix}>0$, only depending on constants in \Unif\,, such that for every $R\geq 1$, every $\mathcal{D}\subset \Lambda_{R} $, every boundary condition $\xi$ on $\mathcal{D}$ and any event $A$ on edges in $\Lambda_{R/2}$, one has
    \begin{equation}
    	c_{mix} \phi^{0}_{\mathcal{D}}[A] \leq \phi^{\xi}_{\mathcal{D}}[A] \leq c_{mix}^{-1} \phi^{0}_{\mathcal{D}}[A].
    \end{equation}
\end{itemize}

\subsection{Continuous correlation functions on isoradial grids}\label{sub:correlation}
Let us now recall the construction of correlation function between points, originally proved in \cite{CHI} and explained in a more systematic way in \cite{CHI-mixed}, including the so-called \emph{fusion rules}, which makes an analytic connection between the OPE in the Conformal Field Theory associated to the critical Ising model and the behaviour of the scaling limits of the planar Ising model when singularities are collapsed to each other. Let us now focus on the definition of the continuous correlation functions attached to the Ising spin field. 

Let $\Omega \neq \mathbb{C}$ be a simply connected domain of the plane, with two marked interior points $v\neq w$. One can then define the double cover $\Omega_{[v,w]}$ of $\Omega$, ramified (with a locally square root type behaviour) around $v$ and $w$. There exist, up to a global sign choice, a unique \emph{holomorphic spinor} $f^{\Omega}_{[v,w]}$ on $\Omega_{[v,w]}$ such that
\begin{itemize}
	\item $h:= \int \Im[(f^{\Omega}_{[v,w]})^2(z)dz]$ extends continuously up to the boundary of $\Omega$ and is constant (say vanishing) along $\partial_{\Omega} $.
	\item One has the following asymptotic
	\begin{equation*}
		f^{\Omega}_{[v,w]}(z)=e^{-i\frac{\pi}{4}}(z-v)^{-\frac{1}{2}}+O((z-v)^{\frac{1}{2}}) \quad \mathrm{ as } \quad z\to v
	\end{equation*}
	and there exist an (a priori unknown) constant $B_{\Omega}(v,w)\in \mathbb{R}$ such that
		\begin{equation*}
		f^{\Omega}_{[v,w]}(z)=B_{\Omega}(v,w)e^{i\frac{\pi}{4}}(z-w)^{-\frac{1}{2}}+O((z-w)^{\frac{1}{2}})\quad \mathrm{ as } \quad z\to w.
	\end{equation*}
It is possible to fix $f^{\Omega}_{[v,w]}$ by requiring that $B_{\Omega}(v,w)\geq 0$. Note that this problem is defined via some Riemann-Hilbert boundary value problem, which means that conformal covariance of the boundary itself allows to prove conformal covariance to the associated solutions. In the present context, for $\varphi:\Omega \to \Omega'$ a conformal map, one has
\begin{equation}
	f^{\Omega}_{[v,w]}(z)=f^{\Omega'}_{\varphi(v),\varphi(w)}(\varphi(z))\cdot (\varphi'(z))^\frac{1}{2}.
\end{equation}
\end{itemize}

To reconstruct the correlation function, one uses the expansion of $f^{\Omega}_{[v,w]}$ near one of its singularities. More precisely, there exist $\mathcal{A}(v,w)\in \mathbb{C}$ such that
	\begin{equation*}
		f^{\Omega}_{[v,w]}(z)=e^{-i\frac{\pi}{4}}(z-v)^{-\frac{1}{2}}+2\mathcal{A}_{\Omega}(v,w)\cdot (z-v)^{\frac{1}{2}}+O((z-v)^{\frac{3}{2}}) \mathrm{ as } z\to v
	\end{equation*}

Then one can \emph{define} the correlation function on $\Omega $ as
\begin{equation}
	\langle \sigma_{u_1} \sigma_{u_2} \rangle_{\Omega}^{(\mathrm{w})}:= \exp \int^{(u_1,u_2)} \Re [\mathcal{A}_{\Omega}(v,w)dv+\mathcal{A}_{\Omega}(w,v)dw],
\end{equation}
 where the overall normalisation is chosen to that $\langle \sigma_{u_1} \sigma_{u_2} \rangle_{\Omega}^{(\mathrm{w})}\sim |u_1 - u_2|^{-\frac14}$ as $u_1 \to u_2 $. In particular one has the conformal rule
 \begin{equation}
 	\langle \sigma_{u_1} \sigma_{u_2} \rangle_{\varphi(\Omega)}^{(\mathrm{w})}=\langle \sigma_{\varphi(u_1)} \sigma_{\varphi(u_2)} \rangle_{\varphi(\Omega)}^{(\mathrm{w})}\cdot|\varphi(u_1)|^{\frac18}\cdot|\varphi(u_2)|^{\frac18}.
 \end{equation}
In the upper half-plane, one has the explicit formula 
\begin{equation*}
	\langle \sigma_a \sigma_{b} \rangle^{+}_{\mathbb{H}} = \frac{2^\frac{1}{2}}{(4\Im[a]\Im[b])^{\frac{1}{8}}} \cdot \Bigg[ \frac{1}{2} \Bigg( \Bigg| \frac{b-\bar{a}}{b-a} \Bigg|^{\frac12} + \Bigg| \frac{b-a}{b-\bar{a}} \Bigg|^{\frac12} \Bigg)  \Bigg]^{\frac12}.
\end{equation*}

\section{Universality of the half-plane one-arm exponent and proof of Theorem \ref{thm:SuperStrongRSW}}

The goal of this section is to use the Kadanoff-Ceva fermionic formalism to establish, through new and comparatively soft arguments, the universality of the one-arm exponent in the half-plane for doubly-periodic graphs. Specifically, we show that on a periodic lattice of mesh size~$1$, the probability that a one-arm event starting from a boundary face of a half-plane discretisation reaches the boundary of a box of size~$n$ decays as $n^{-1/2}$, up to multiplicative constants (both upper and lower) depending only on the parameters in \Unif. To achieve this, we introduce a new method for controlling the boundary behaviour of Kadanoff-Ceva fermions near straight cuts. This approach can also be applied to prove the conformal invariance of FK martingale observables in smooth domains with sufficiently regular discretisations \cite[Chapter~6]{MahPHD}. Once the universality of the one-arm exponent is established, we incorporate it into the framework of \cite{DMT21}, together with standard surgery techniques. This yields a proof of Theorem~\ref{thm:SuperStrongRSW}, while circumventing the need for $\pi/2$-rotation invariance and self-duality, which play a central role in the classical arguments.
\subsection{Identification of the one-arm exponent in the half-plane}

Fix a face $f \in \cS^\delta$, and let $z_f$ be the center of its tangential circle. Denote by $\mathbb{H}^\delta(f) \subset \cS^\delta$ an approximation (by the grid $\cS^\delta$ up to $10\delta$) of the half-plane $z_f + \mathbb{H} \subset \mathbb{C}$. By \emph{discretisation}, we mean that the boundary of $\mathbb{H}^\delta(f)$, viewed as a polygonal line in $\mathbb{C}$, lies within Hausdorff distance at most $10\delta$ from the line $z_f + \mathbb{R}$. When computing one-arm exponents, the precise choice of this discretisation will be irrelevant, and all results hold uniformly over such choices. For simplicity, we assume that the boundary of $\mathbb{H}^\delta(f)$ is periodic in the horizontal direction. 

In the infinite-volume limit, the FK-Ising model can be defined on $\mathbb{H}^\delta(f)$, with weights inherited from those of $\cS^\delta$, and with \emph{free} boundary conditions along $\partial \mathbb{H}^\delta(f)$. By Theorem~\ref{thm:RSW-FK}, this free measure is unique; denote it by $\mathbb{P}_{\mathbb{H}^\delta(f)}$. Next, define the square $\Lambda_R(z_f) \subset \mathbb{C}$ of width $2R$, centered so that the midpoint of its bottom side is $z_f$. We then fix (again up to $10\delta$) a discretisation $\Lambda^\delta_R(f)$ of $\Lambda_R(z_f)$, assuming for simplicity that the bottom boundary of $\Lambda^\delta_R(f)$ is contained in $\partial \mathbb{H}^\delta(f)$. This setup allows us to define the following event for the FK-Ising model:
\begin{equation*}
	\Big \{ f \overset{\mathbb{H}^\delta(f)}{\longleftrightarrow} \partial \Lambda^\delta_{R}(f)\Big\},
\end{equation*}
This event asserts the existence of a cluster of open edges in $\Lambda^\delta_{R}(f)$ connecting $f$ to the left, right, or top boundary of $\Lambda^\delta_{R}(f)$. For doubly-periodic graphs, the probability of this event corresponds precisely to the one-arm exponent in the half-plane starting from $f$. By the finite energy property and translation invariance of the boundary of $\mathbb{H}^\delta(f)$, this one-arm exponent is identical up to a universal multiplicative constant for all faces of $\mathcal{T}$. We are now ready to state the main proposition of this section.
\begin{prop}\label{prop:one-arm-half-plane}
	There exist positive constants $C,L$, only depending on constants in \Unif\ such that for any $R\geq L\cdot \delta  $
	\begin{equation}
		C^{-1} \Big(\frac{\delta}{R}\Big)^{\frac12} \leq  \mathbb{P}_{\mathbb{H}^\delta(f)}\Big[  f \overset{\mathbb{H}^\delta(f)}{\longleftrightarrow} \partial \Lambda^\delta_{R}(f) \Big] \leq C \Big(\frac{\delta}{R}\Big)^{\frac12}.
	\end{equation} 
\end{prop}

This result is known on isoradial lattices from \cite[Section~5.1]{ChSmi2}, where it follows from a comparison between the primitive $H[X^{\mathrm{FK}}_{\Omega^\delta}]$ and discrete harmonic functions with Dirichlet boundary conditions near straight arcs. The argument relies on the sub/super-harmonicity of $H[X^{\mathrm{FK}}_{\Omega^\delta}]$ on the two halves of the lattice, together with sharp estimates of discrete harmonic measures. In the periodic case, however, this approach no longer applies, since $H[X^{\mathrm{FK}}_{\Omega^\delta}]$ is \emph{not} known to be harmonic for any simple local Laplacian.

We begin by establishing the upper bound stated in the next proposition, using new and soft arguments to control the growth near straight arcs of primitives of $s$-holomorphic functions. This framework was first introduced in \cite[Chapter~6]{MahPHD} and is presented here for the first time in a research article. To simplify the exposition, we first show that the probability of a one-arm event in a half-plane at distance $n$ decays as $n^{-1/2}$ by proving an analogous, and computationally simpler, statement in a rectangle. In this setting (see Figure \ref{fig:FK-proof}), the free arc includes half of the vertical sides of $\Lambda^\delta_{R}(f)$, while the remaining boundary segment above the line $y = \Im[z_f] + R$ is wired. More precisely, fix a corner $a^\delta \in \partial \Lambda_{R^\delta}(f)$ located at $(\Re[z_f] - R, \Im[z_f] + R) + O(\delta)$ and a corner $b^\delta \in \partial \Lambda_{R^\delta}(f)$ located at $(\Re[z_f] + R, \Im[z_f] + R) + O(\delta)$, where $O$ is a uniform constant smaller than $10$.

We consider the FK-Ising model on $\Lambda^\delta_{R}(f) \cap \mathbb{H}^\delta(f)$ with free boundary conditions on the bottom arc $(a^\delta b^\delta)^\bullet$ and wired boundary conditions on the complementary arc $(b^\delta a^\delta)^\circ$. For simplicity, we continue to assume that the bottom discretisation of $\Lambda^\delta_{R}$ is contained in $\partial \mathbb{H}^\delta(f)$. In the next lemma, we denote by $\mathbb{P}^{\circ,\bullet}_{\Lambda^\delta_{R}}$ the FK-Ising  measure on $\Lambda^\delta_{R}$ with these boundary conditions. This allows us to define the event
\begin{equation*}
	\Big \{ f \overset{\bullet,\circ}{\longleftrightarrow} \partial \Lambda^\delta_{R}(f) \Big\},
\end{equation*}
encoding the fact that $f$ is connected by a path of open edges to the boundary arc $(b^\delta a^\delta)^\circ$. The following lemma is the building block to prove Proposition \ref{prop:one-arm-half-plane}.
\begin{lemma}
		There exist positive constants $C,L$, only depending on the constants in \Unif\ such that for any $R\geq L\cdot \delta  $
	\begin{equation}
C^{-1} \cdot \Big(\frac{\delta}{R}\Big)^{\frac12} \leq \mathbb{P}^{\circ,\bullet}_{\Lambda^\delta_{R}}\Big[ f \overset{\bullet,\circ}{\longleftrightarrow} \partial \Lambda^\delta_{R}(f) \Big] \leq C \cdot \Big(\frac{\delta}{R}\Big)^{\frac12}.
	\end{equation} 
\end{lemma}

\begin{proof}
We recommend that the reader follow the argument with the help of Figure \ref{fig:FK-proof}. Throughout the proof, we assume that the ratio $(R/\delta)$ is sufficiently large; otherwise, the same conclusion follows by invoking the finite-energy property of the FK model, at the cost of adjusting the constant $C$.  By rescaling all distances by $R$, we may assume without loss of generality that $R = 1$. We now introduce some notation. Let $X^\delta$ denote the FK observable in $\Lambda^\delta_{R}$, and let $F^\delta$ be the associated $s$-holomorphic function defined by \eqref{eq:F-from-X}. As recalled above, up to an appropriate choice of additive constant, one has 
\[
H_{X^\delta}(v^\circ)=0 \quad \text{for all } v^\circ \in (b^\delta a^\delta)^\circ,
\qquad 
H_{X^\delta}(v^\bullet)=1 \quad \text{for all } v^\bullet \in (a^\delta b^\delta)^\bullet.
\]

\textbf{Proof of the upper bound using the comparison principle} We are going to compare $H_{X^\delta}$ with the primitive of a constant $s$-holomorphic function.
Define the \emph{constant} $s$-holomorphic function $F_{\textrm{comp}}:=2e^{-i\frac{\pi}{4}}$ and denote by 
\begin{equation*}
	H_{\textrm{comp}}:=H[F_{\textrm{comp}}]= -4\int \Im[d\cS^\delta+ d\cQ^\delta] \, .
\end{equation*}
Up to a suitable choice of additive constant, we may assume (since $\cQ^\delta$ is periodic) that
\[
H_{\textrm{comp}}=1-O(\delta)
\quad\text{along the horizontal bottom boundary of }\Lambda_{R}^\delta,
\]
On the other hand, the specific choice of $F_{\textrm{comp}}$ and the periodicity of $\cS^\delta$ ensure that $H[F_{\textrm{comp}}]$ grows \emph{linearly at a negative speed} when traveling in the upward direction.  
In particular, for sufficiently large $O(\delta)$ depending only on the constants in \Unif, one has  
\[
H[F_{\textrm{comp}}](v)\leq -\frac{1}{2}
\quad \text{for any } v \text{ such that } 
\Im[v]\geq \Im[z_f] + 1 + O(\delta).
\]

We now compare the boundary values of $H_{\textrm{comp}}$ and $H_{X^\delta}$ along $\partial \Lambda_{1^\delta}$:

\begin{itemize}
	\item \textbf{Bottom free boundary:}  
	On the bottom free boundary of $\Lambda^\delta_{1}$, we have $H_{X^\delta}(v^\bullet) = 1$ and $H_{\textrm{comp}} = 1 + O(\delta)$.  
	Hence,
	\[
	H_{X^\delta} - H_{\textrm{comp}} \ge O(\delta)
	\quad \text{along the bottom free boundary of } \Lambda^\delta_{1}.
	\]
	
	\item \textbf{Vertical free arcs:}  
	On the vertical free arcs of $\Lambda^\delta_{1}$, we have $H_{X^\delta}(v^\bullet) = 1$ and $H_{\textrm{comp}} \le 1 + O(\delta)$.  
	Thus,
	\[
	H_{X^\delta} - H_{\textrm{comp}} \ge O(\delta)
	\quad \text{along the vertical free boundary arcs of } \Lambda^\delta_{1}.
	\]
	
	\item \textbf{Wired arc:}  
	On the wired arc of $\Lambda^\delta_{1}$, we have $H_{X^\delta}(v^\circ) = 0$ and $H_{\textrm{comp}} \le 0$.  
	Therefore,
	\[
	H_{X^\delta} - H_{\textrm{comp}} \ge 0
	\quad \text{along the wired boundary arc of } \Lambda^\delta_{1}.
	\]
\end{itemize}

Therefore, we can apply the \emph{comparison principle}, stated in \cite[Proposition~2.11]{Che20}, which guarantees that the inequality \(H_{X^\delta} - H_{\textrm{comp}} \ge O(\delta)\) along \(\partial \Lambda^\delta_{1}\) also holds in the interior, that is, \(H_{X^\delta} - H_{\textrm{comp}} \ge O(\delta)\) \emph{throughout} \(\Lambda^\delta_{1}\). A subtlety remains: for instance, if a vertex \(v^\circ\) lies directly adjacent to \((b^\delta a^\delta)^\bullet\), it may happen that 
\[
H_{X^\delta}(v^\circ) - H_{X^\delta}((b^\delta a^\delta)^\bullet) = -X(c_{(b^\delta a^\delta)^\bullet v^\circ})^2 \ll \delta.
\]
Nevertheless, repeating verbatim the proof of \cite[Proposition~2.11]{Che20} shows that along the free arc, such a vertex \(v^\circ\) \emph{cannot} be a local minimum of \(H_{X^\delta} - H_{\textrm{comp}}\), since the associated fermions branch there-unlike the vertices \(v^\bullet\) on the arc \((b^\delta a^\delta)^\bullet\). An analogous argument applies along the wired arc. Consequently, one indeed obtains that \(H_{X^\delta} - H_{\textrm{comp}} \ge O(\delta)\) \emph{everywhere inside} \(\Lambda^\delta_{1}\).

This argument suffices to establish the upper bound. To see this, fix a sufficiently large constant \(M > 0\) (its exact value will be determined shortly).  
There exists a universal constant \(C_0\), depending only on the parameters in \Unif, such that by \cite[Theorem~2.18]{Che20}, for any \(r \ge M\delta\) large enough and any \(z \in B(z, r/2)\),
\begin{equation}\label{eq:gradient-H}
	|F^\delta(z)|^2 \le C_0 \frac{\textrm{osc}_{B(z_f + ir, r)} H_{X^\delta}}{r},
\end{equation}
where $\textrm{osc}_{B}$ denotes the oscillation (the difference between the maximum and the minimum) of $H_{X^\delta}$ over the ball $B$.  

Taking \(r = 2M\delta\) with \(M\) large enough, there exist a constant $C_2=C_2(M)$, depending on $M$ and constants in \Unif\, such that for \(v \in B(z_f + 4iM\delta, M\delta)\) 
\begin{equation*}
	1 \ge H_{X^\delta}(v) \ge H_{\textrm{comp}}(v) - O(\delta) \ge 1 - C_2\delta,
\end{equation*}
since \(H_{\textrm{comp}}\) decreases linearly (with a slope bounded away from \(0\) and \(\infty\)) when moving upward from the bottom boundary of \(\Lambda_{1^\delta}\), where its value is \(1 + O(\delta)\).  Recalling that $H_{X^\delta}\leq 1$, this implies that  \(\textrm{osc}_{B(z_f + 4iM\delta, M\delta)} H_{X^\delta} \leq C_2 \cdot \delta\).  
Applying \eqref{eq:gradient-H} to this ball gives
\begin{equation}
	|F^\delta(z)|^2 \le C_0 \frac{C_2 \cdot \delta}{M\delta} = O(1),
\end{equation}
for some uniform constant \(O(1)\) depending only on parameters in \Unif.  
Hence, \(F^\delta(z)\) is \emph{uniformly bounded up at a distance} \(4M\delta\) from the boundary.  

Using the reconstruction formula \eqref{eq:F-from-X} and the fact that all angles are bounded away from \(0\) and \(\pi\), while all edge lengths are comparable to \(\delta\), we deduce that \(\delta^{-1/2} X^{\delta}\) is uniformly bounded at a distance \(4M\delta\) from the bottom boundary of \(\Lambda_{1^\delta}\).  
Since all coupling constants in $\cS^\delta$ are bounded away from \(0\) and \(\infty\) (see \eqref{eq:theta-from-S}), the finite-energy property and the fact that one only a bounded number of steps to reach the boundary of \(\Lambda_{1^\delta}\)) from \(B(z_f + 4iM\delta, M\delta)\), this ensures that \(\delta^{-1/2} X^\delta\) remains uniformly bounded \emph{up to the boundary}.  For a boundary corner \(c = ((a^\delta b^\delta)^\bullet, v^\circ)\) attached to the free arc, one has
\begin{equation*}
	X^{\delta}(c) = \mathbb{E}^{\circ,\bullet}_{\Lambda_{1}^{\delta}}[\sigma_{v^\circ}\sigma_{(b^\delta a^\delta)^\circ}] = \mathbb{P}^{\circ,\bullet}_{\Lambda_{1}^\delta}\big[\, f \overset{\bullet,\circ}{\longleftrightarrow} \partial \Lambda_{1}^\delta(f)\, \big] = O(\delta^{1/2}).
\end{equation*}
This concludes the proof of the upper bound.

\medskip
\textbf{Proof of the lower bound by contradiction}
Assume once again that $R=1$ (up to scaling the lattice) and there exist a sequence $\delta_n \to 0 $ such that 
\begin{equation}\label{eq:symmetry-convergence-to-0}
	|\mathbb{E}^{\circ,\bullet}_{\Lambda^{\delta_n}_{1}}[\sigma_{v^\circ}\sigma_{(b^\delta a^\delta)^\circ}]|\cdot  \delta_n^{-1/2} \longrightarrow 0
\end{equation}
as $n\to \infty$. Note that using Theorem \ref{thm:RSW-FK}, periodicity of the lattice the boundary of $\Lambda_1^{\delta_n}$ and the finite energy property, one can conclude that uniformly on faces $v^\circ=f'\in \partial \Lambda^{\delta_n}_{1} \cap \partial \mathbb{H}^\delta(f) \cap \{ |\Re[z_f-z_{f'}]|\leq \frac{1}{4}  \}$ one has as $\delta_n \to 0 $
\begin{equation}\label{eq:symmetry-convergence-to-0}
	|\mathbb{E}^{\circ,\bullet}_{\Lambda^{\delta_n}_{1}}[\sigma_{v^\circ}\sigma_{(b^\delta a^\delta)^\circ}]|\cdot  \delta_n^{-1/2} \longrightarrow 0.
\end{equation}
By the strong box crossing property (Theorem~\ref{thm:RSW-FK}), uniformly in $\delta$ there exists a universal constant (depending only on that theorem) such that the probability of a one-arm event in the half-plane from $f$ up to distance $1/8$ is uniformly comparable to the probability of an arm event in a discretised square whose bottom side lies on $\partial \mathbb{H}^\delta(f)$ and is centred at $f \in \partial \mathbb{H}^\delta(f)$, regardless of the boundary conditions imposed on the vertical sides and the top side. By periodicity of the lattice and the boundary discretization, this implies \eqref{eq:symmetry-convergence-to-0}. We claim that this already leads to a contradiction. The proof is organised in several steps below.
\begin{itemize}
	\item As recalled in Section~\ref{sub:HF-def}, the $s$-holomorphic functions $(F^{\delta_n})_{n\ge 1}$ are uniformly bounded on compact subsets of $\Lambda_1$ and admit a subsequential limit $f$ as $n \to \infty$. One of the main results of \cite{Che20} further guarantees that this limit exists (without passing to a subsequence) and that $f$ is a \emph{non-trivial} holomorphic function, which can in fact be identified with the solution to a certain Riemann--Hilbert boundary value problem.

	\item At the discrete level in our setting, the functions $(\Re[F^{\delta_n}])_{n\ge 0}$ are harmonic on the $S$-graphs $(\cS^{\delta_n} - i\cQ^{\delta_n})_{n\ge 0}$, for the Laplacian defined in \cite[Section~2.5]{Che20}. In particular, the associated random walks satisfy \emph{uniform crossing estimates above scale} $\delta$: for all $r \ge \delta$, the probability that a random walk makes a complete turn within the annulus of radii $r$ and $2r$ is uniformly bounded from below. 

	Consider the point $z_f + is$ for some macroscopic (but sufficiently small) $s > 0$. Recall that, as $n \to \infty$, $(\Re[F^{\delta_n}])_{n\ge 0}$ vanishes near the middle of the arc $\Lambda^{\delta_n}_{1} \cap \partial \mathbb{H}^\delta(f)$, restricted to $\{\, |\Re[z_f - z_{f'}]| \le 1/4 \,\}$. As $s \to 0$, one can use weak Beurling estimates (see \cite[Proposition~2.11]{ChSmi1}) to reconstruct $\Re[F^{\delta_n}(z_f + is)]$ from the boundary values of $\Re[F^{\delta_n}(z)]$ (which vanish in the limit) and an error term that decays polynomially in $s$. More precisely (see, e.g., \cite[Step~2 in Section~4]{Mah23} or \cite[Theorem~3.12]{ChSmi1}), there exist universal constants $\beta > 0$ and $O$, depending only on the parameters in \Unif, such that
	\begin{equation*}
		\Re[F^{\delta_n}(z_f + is)] = o_{\delta_n \to 0}(1) + O(s^\beta).
	\end{equation*}
	A similar argument shows that, uniformly in $\delta$ and $s$,
	\begin{equation*}
		\Re[F^{\delta_n}(z_{f'} + is)] = o_{\delta_n \to 0}(1) + O(s^\beta),
	\end{equation*}
	for any face $f' \in \partial \Lambda^{\delta_n}_{1} \cap \partial \mathbb{H}^\delta(f)$ satisfying $|\Re[z_f - z_{f'}]| \le 1/8$.

	\item We can now conclude. Sending first $\delta_n \to 0$ and then $s \to 0$, we obtain that $\Re[f]$ vanishes along the boundary segment $z_f + [-1/8, 1/8]$. By the same reasoning, $\Im[f]$ also vanishes along this segment. Hence, $f$ vanishes on a nontrivial boundary interval of $\Lambda_{1}$, and therefore $f \equiv 0$ throughout the domain, contradicting the fact that $f$ is a nontrivial holomorphic function. This contradiction completes the proof.
\end{itemize}
\end{proof}

\begin{figure}
\begin{minipage}{0.49\textwidth}
\includegraphics[clip, width=1\textwidth]{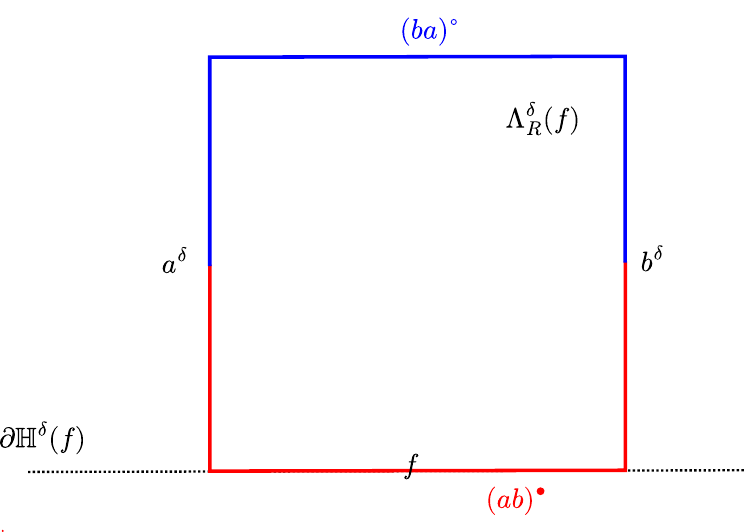}
\end{minipage}\begin{minipage}{0.49\textwidth}
\includegraphics[clip, width=1.1\textwidth]{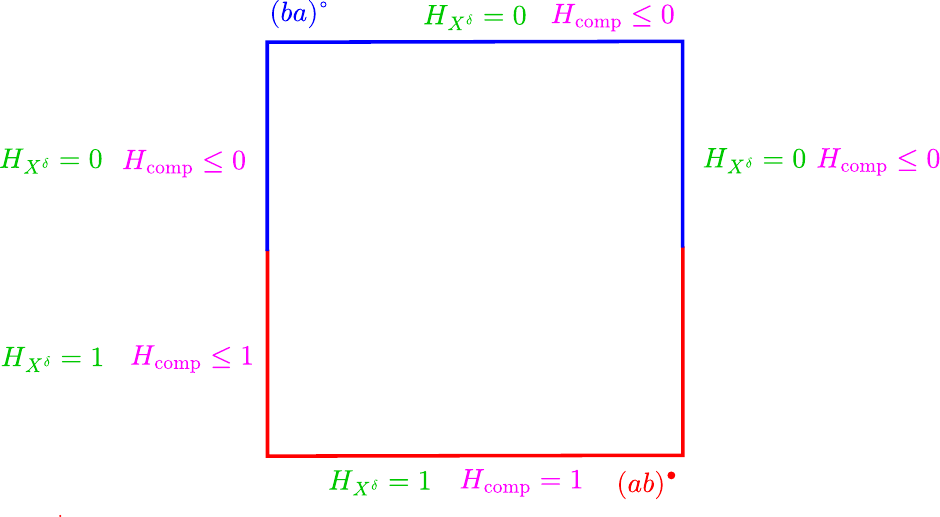}
\end{minipage}
\caption{Left: The wired arc $(b^\delta a^\delta)^\circ$ is drawn in blue while the free arc $(a^\delta b^\delta)^\bullet$ is drawn in red. Right: Boundary values of $H_{X^\delta}$ and $H_{\textrm{comp}}$ used for the comparison principle.}
\label{fig:FK-proof}
\end{figure}

\begin{proof}[Proof of Proposition \ref{prop:one-arm-half-plane}]
It suffices to adapt the ideas underlying the mixing property recalled in Section~\ref{sub:basic-properties-FK}. To simplify the exposition, we present the proof of the upper bound through a sequence of diagrams, where each diagram represents an event whose probability increases (up to a uniform multiplicative constant) when moving from one configuration to the next. In the illustrations, primal edge clusters are shown in blue, while dual clusters are shown in red.
\begin{figure}
\begin{minipage}{0.32\textwidth}
\includegraphics[clip, width=0.8\textwidth]{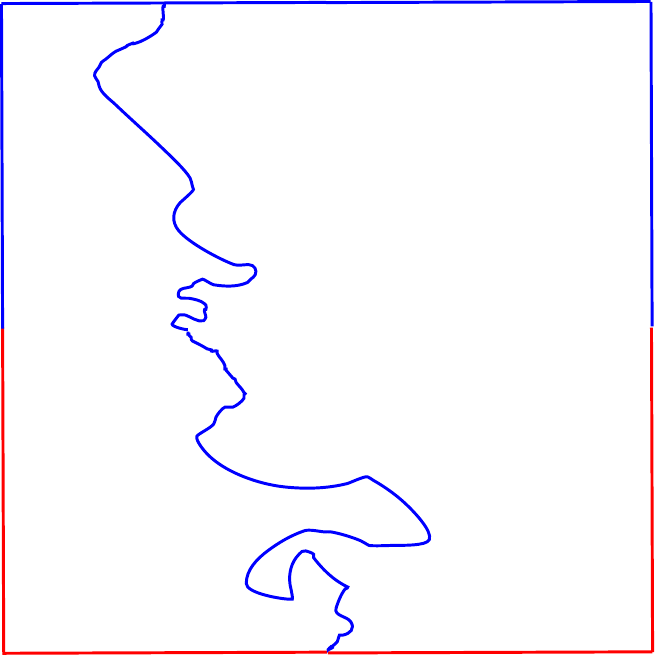}
\end{minipage}\hskip 0.02\textwidth \begin{minipage}{0.33\textwidth}
\includegraphics[clip, width=0.8\textwidth]{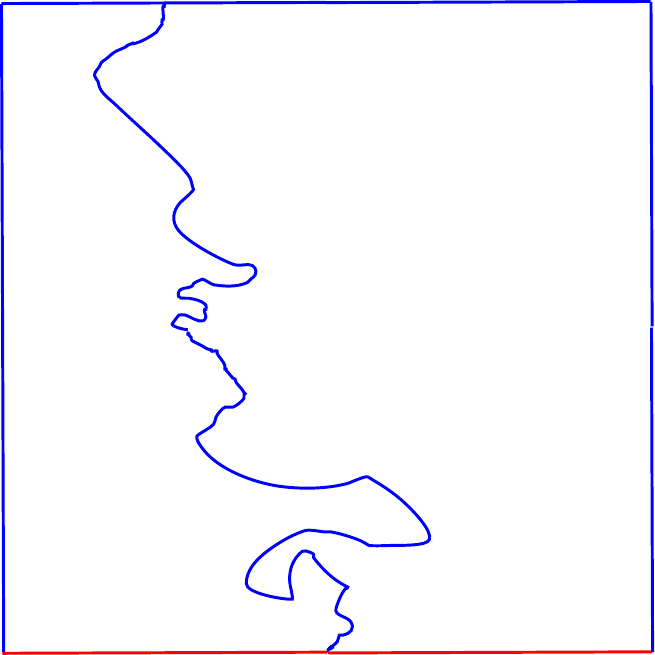}
\end{minipage}
\begin{minipage}{0.32\textwidth}
\includegraphics[clip, width=0.8\textwidth]{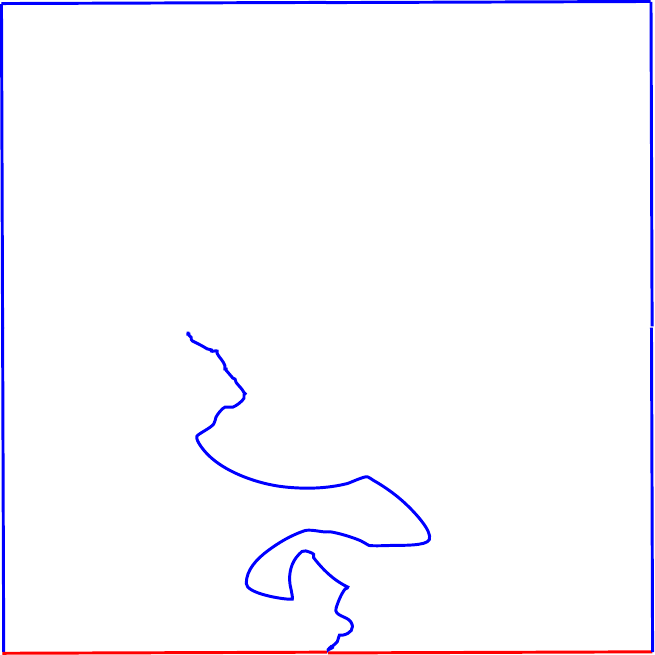}
\end{minipage}\hskip 0.02\textwidth \begin{minipage}{0.33\textwidth}
\includegraphics[clip, width=0.8\textwidth]{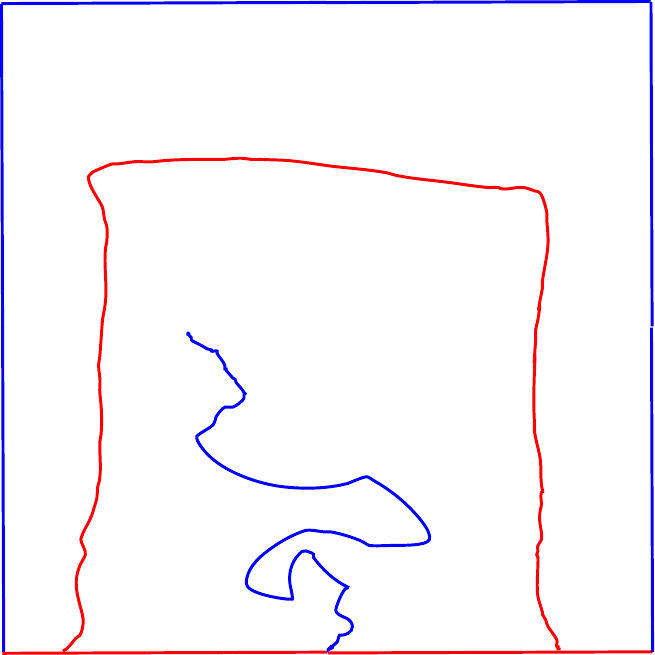}
\end{minipage}
\begin{minipage}{0.32\textwidth}
\includegraphics[clip, width=\textwidth]{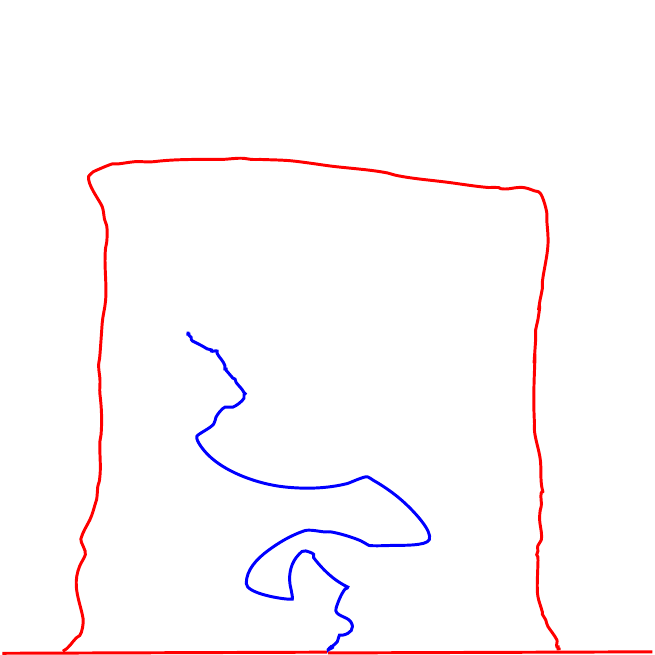}
\end{minipage}\hskip 0.02\textwidth \begin{minipage}{0.32\textwidth}
\includegraphics[clip, width=0.8\textwidth]{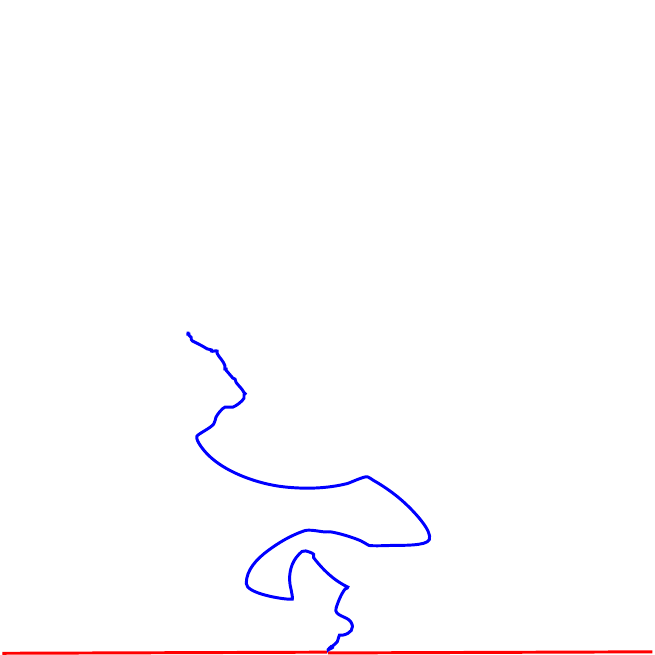}
\end{minipage}
\caption{First line: (Left) Arm event up the boundary of $\Lambda^\delta_R $ with alternating wired boundary conditions on top half of the square and free boundary conditions on the bottom half of the square. (Middle) Arm event up the boundary of $\Lambda^\delta_R $ with alternating free boundary conditions on the bottom segment and wired boundary conditions on the rest of $\partial \Lambda^\delta_R$. (Right) Arm event up the boundary of $\Lambda^\delta_{R/2} $ with alternating free boundary conditions on the bottom segment and wired boundary conditions on the rest of $\partial \Lambda^\delta_R$. Second line: (Left) Arm event up the boundary of $\Lambda^\delta_{R/2} $ with alternating free boundary conditions on the bottom segment and wired boundary conditions on the rest of $\partial \Lambda^\delta_R$, imposing additionally a dual wired circuit separating $\Lambda^\delta_{R/2}$ and $\Lambda^\delta_R$. (Middle) Arm event up the boundary of $\Lambda^\delta_{R/2} $ with an additional a dual wired circuit separating $\Lambda^\delta_{R/2}$ and $\Lambda^\delta_R$. (Right) Arm event up the boundary of $\Lambda^\delta_{R/2} $ in the half-plane with free boundary conditions.}
\end{figure}.
All the events are increasing order (using monotonicity of boundary conditions), except when passing from the third to the fourth configuration. Still, the existence of an open circuit is a decreasing event while the arm event is an increasing event. One can apply the FKG inequality (between an increasing and a decreasing event) and use the fact that the probability to have a dual separating circuit between $\Lambda^\delta_{R/2}$ and $\Lambda^\delta{R}$ is bounded away from $0$ and $1$. Hence, the probabilities increase (up to constant) for this sequence of events. The proof of the lower bound can be done with similar techniques.
\end{proof}

\subsection{From the half-plane one-arm exponent to Theorem \ref{thm:SuperStrongRSW}}
In this section, we show that the precise computation of the one-arm exponent in the half-plane is sufficient to deduce Theorem~\ref{thm:SuperStrongRSW}, following the approach of \cite{DMT21}. We emphasize that without the explicit knowledge of this exponent, the same renormalisation argument cannot be carried out: the proof in \cite{DMT21} relies on estimates involving both primal and dual arm exponents. For the critical square lattice, self-duality compensates for this limitation, while in our setting, the exact value of the one-arm exponent allows us to reach the same conclusion.

We work on the original periodic grid $\cS$, assuming that all edges have length comparable to $1$. Compared to the square-lattice case of \cite{DMT21}, controlling the boundary argument of fermionic observables along the horizontal and vertical directions of wired arcs is considerably more delicate. To overcome this, we introduce an explicit \emph{extension} procedure of the domain using kites attached along the boundary. This construction allows us to fix the boundary argument of fermionic observables in a slightly enlarged domain, which suffices for the proof. We focus here on describing this extension procedure (conceptually related to \cite{Mah23}) and on the adaptations required in the steps of \cite{DMT21}; all other ingredients are quoted directly from that work. For clarity, we retain the same notation as in \cite{DMT21}, modifying it only as needed to accommodate the case of critical doubly-periodic graphs.

Fix a face type $f_0 \in \mathcal{T}$ and assume that the origin corresponds to the center of a face $z_{f_0}$. We identify the translates $\mathcal{G}_{(i,j)}$ of the fundamental domain $\mathcal{G}$ of $\cS$ with the integer lattice $\mathbb{Z}^2$. Define
\begin{equation*}
	\Lambda_n := \bigcup_{|i|, |j| \le n} \mathcal{G}_{(i,j)},
\end{equation*}
the union of $n^2$ fundamental domains centered at $z_{f_0}$. For $x = (i,j) \in \mathbb{Z}^2$, denote by $\Lambda_n(x)$ the translate of $\Lambda_n$ centered at the face of $\mathcal{G}_{(i,j)}$ having the same type as $f_0$. For $B := \Lambda_r(x)$, let $\overline{B}$ be the twice larger box with the same center.  

Let $\mathcal{D}$ denote the annular region between $B$ and $\overline{B}$. Then, by Theorem~\ref{thm:RSW-FK}, there exists $c_0 > 0$ such that
\[
\phi^{0}_{\mathcal{D}}\big[\text{there exists a circuit of open edges in } \mathcal{D}\big] > c_0.
\]
Informally, in the present setting, we replace the notion of vertices by that of fundamental domains. We say that $\mathcal{D}$ is \emph{$R$-centred} if it contains $\Lambda_{2R}$ but not $\Lambda_{3R}$. The next proposition, adapted from \cite{DMT21}, is one of the main ingredients in the proof of Theorem~\ref{thm:SuperStrongRSW}.
\begin{prop}[Proposition 1.3 in \cite{DMT21}]
	There exist $c>0$ such that for nay $R>1$ and any $R$-centred domain one has
	\begin{equation*}
		\phi^{0}_{\mathcal{D}}[\Lambda_{R} \overset{\Lambda_{9R}}{\longleftrightarrow} \partial \mathcal{D}]\geq c
	\end{equation*}
\end{prop}
We now introduce additional notation. For $r \ge 0$, an \emph{$r$-box} is defined as a translate of $\Lambda_r$ by a vector $x = (i, j)$ in $(1 \vee r)\mathbb{Z}^2$. In particular, $0$-boxes correspond to fundamental domains.  Given $R \ge 1$ and an $R$-centred domain $\mathcal{D}$, let $\mathbf{M}_r(\mathcal{D}, R)$ denote the number of $r$-boxes intersecting $\partial \mathcal{D}$. In particular, $\mathbf{M}_0(\mathcal{D}, R)$ counts the number of fundamental domains intersecting $\partial \mathcal{D}$ that are connected to $\Lambda_R$ within $\mathcal{D} \cap \Lambda_{7R}$. The goal is to prove the existence of a constant $c > 0$ such that
\begin{equation*}
	\phi^{0}_{\mathcal{D}}\big[\mathbf{M}_0(\mathcal{D}, R) \ge 1\big] \ge c.
\end{equation*}

In the present analysis (as compared to \cite{DMT21}), vertices are replaced by fundamental domains. Nevertheless, by the finite energy property, this substitution does not affect the reasoning or the estimates (up to universal constants). Indeed, one can always enforce that all edges within a given fundamental domain are either open or closed, loosing only a multiplicative factor depending on the constants in \Unif.

Following \cite[(1.5)]{DMT21}, a key step is to establish a lower bound on the first moment of $\mathbf{M}_r(\mathcal{D}, R)$. For this purpose, define
\begin{equation}
	M(r, R) := \inf\big\{\phi^{0}_{\mathcal{D}}[\mathbf{M}_0(\mathcal{D}, R)] : \mathcal{D} \text{ is } R\text{-centred} \big\}.
\end{equation}
The following result provides a uniform (though non-sharp) lower bound on this first moment.

\begin{prop}[Proposition~1.4 in \cite{DMT21}]\label{prop:1.4-DMT}
	There exists a constant $c_1 > 0$ such that for every $R \ge r \ge 1$,
	\begin{equation*}
		M(r, R) \ge c_1 (R/r)^{c_1}.
	\end{equation*}
\end{prop}
In fact, a stronger version of this estimate holds, involving one-arm exponents in the half-plane evaluated at different scales. For simplicity, set
\begin{equation*}
	\pi_{1}^{+}(R) := \phi^{0}\big[ f_0 \overset{\mathbb{H}(f_0)}{\longleftrightarrow} \partial \Lambda_{R}(f_0) \big].
\end{equation*}
This leads to the main proposition used in the multi-scale analysis.
\begin{prop}[Proposition~1.8 in \cite{DMT21}]\label{prop:M(r,R)}\label{prop:1.8-DMT}
	There exists a constant $c_4 > 0$ such that for every $R \ge r \ge 0$,
	\begin{equation*}
		M(r, R) \ge c_4 \frac{R\,\pi_{1}^{+}(R)}{1 \vee r\,\pi_{1}^{+}(r)}.
	\end{equation*}
\end{prop}
Inserting the exact value of the one-arm exponent in the half-plane obtained in the previous section allows one to deduce Proposition~\ref{prop:1.8-DMT} directly from Proposition~\ref{prop:1.4-DMT}. Moreover, the argument in \cite[Section~5.1]{DMT21} shows that combining Theorem~\ref{thm:RSW-FK} with this exponent reduces the analysis to the case $r = 0$.  
This case will be handled through a key observation: the boundary contour integral of the discrete FK observable \emph{almost} vanishes, while its boundary values precisely encode the primal and dual correlations.  

On periodic graphs, however, constructing suitable discretisations of the half-plane is subtler. In particular, we will instead work with an \emph{extension} of the original domain, obtained by adding an extra layer of boundary kites. In this extended domain, the boundary argument of the associated FK observable becomes tractable.  We now introduce some additional notation. For $\ell, m \ge 0$, we call $(\Omega, a, b)$ an \emph{$(m,\ell)$-corner Dobrushin domain} if its boundary consists of:
\begin{itemize}
	\item an approximation of the horizontal segment between $(0,0)$ and $b := (0,m)$ along $\partial \mathbb{H}$;
	\item an approximation of the vertical segment between $(0,0)$ and $a := (\ell,0)$ along $\partial i\mathbb{H}$;
	\item a self-avoiding curve $\gamma$ from $a$ to $b$, avoiding the two segments above and oriented clockwise around $0$.
\end{itemize}

Note that the notion of a \emph{corner Dobrushin domain} here is unrelated to the \emph{Kadanoff--Ceva corners} appearing in the fermionic formalism, which we will explicitly refer to as such below.  We now state the following lemma, analogous to \cite[Lemma~5.3]{DMT21}, which plays a key role in the renormalisation argument of \cite{DMT21}. 

\begin{lemma}\label{lem:lower-bound-crossing}
	There exists $c_1 > 0$ such that for every $(m,\ell)$-corner Dobrushin domain $\Omega$ with $m \ge \ell$, one has
	\begin{equation}
		\sum_{v^\circ \in \gamma} \phi_{\Omega}^{0/1}\big[v^\circ \longleftrightarrow (ab)^\circ\big] \ge c_1\, m^{\frac{1}{2}}.
	\end{equation}
\end{lemma}

We will not prove this lemma directly but instead establish the same bound for a slightly enlarged version of the domain.

\textbf{Extension of the domain to construct special cuts made of kites}
A central tool from \cite[Section~5.2]{Che20} is the notion of \emph{discrete half-planes}. These give a convenient discretisation of standard half-planes in which the boundary values of discrete fermionic observables have a prescribed sign, a property that is crucial here. They allow us to track boundary values along a given arc, but are not directly suited to controlling boundary values simultaneously on discretisations of orthogonal arcs.

We work with $\cS$ periodic. \emph{For simplicity, assume that one axis of the torus attached to the fundamental domain of $\cS$ is horizontal.} The general case follows by minor adjustments to the definition of $(m,\ell)$-corners. We implement the ideas of \cite{Mah23} and \cite[Chapter~6]{MahPHD}, relying on detailed arguments developed in \cite[Section~3]{Mah23}. The construction proceeds as follows (see also Figure~\ref{fig:kite-extension}):

\begin{itemize}
	\item In the plane $\cS$, fix the horizontal line $\mathcal{L}$ at level $y=0$, chosen so that it avoids all vertices of $\cS$. By \cite[Definition~3.1]{Mah23}, one can modify the tangential quadrilaterals intersecting $\mathcal{L}$ so that $\cS \cap \mathbb{H}$ becomes a half-plane that remains a proper $s$-embedding, with boundary vertices of $G^\bullet$ and $G^\circ$ lying on $\cS \cap \mathcal{L}$.
	\item Extend each boundary vertex by attaching a triangle, viewed as a \emph{tangential quadrilateral}, whose inscribed circle touches $\mathcal{L}$ exactly at black vertices.
	\item Below this triangular layer, attach a layer of \emph{kites}. In particular, the vertices of $G^\circ$ of these kites are placed so as to maintain the required $s$-embedding structure and to control the boundary argument of the associated FK observable.
	\item If $\cS$ is periodic, one can choose the cutting level $y$ (after a small vertical adjustment) so that all quads of $\cS \cap \mathbb{H}$ produced by this construction satisfy a uniform \Uniffff\ property, with constants depending only on those of the original grid $\cS$. This is where the horizontal axis of the torus is used.
\end{itemize}

The last point follows readily from the proof of \cite[Proposition~3.4]{Mah23}. A similar vertical construction is possible, although the very last layer may fail to remain \Uniffff.
\begin{figure}
\begin{minipage}{0.325\textwidth}
\includegraphics[clip, width=1.5\textwidth]{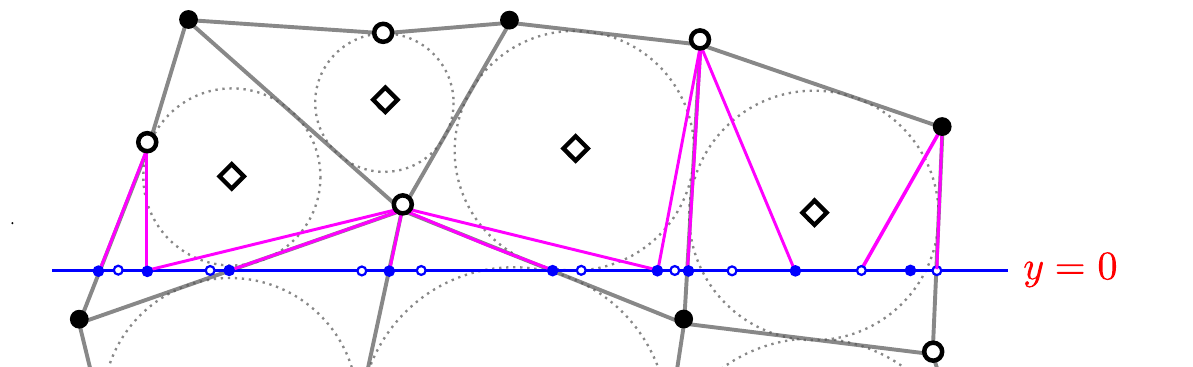}
\end{minipage} \hskip 0.07\textwidth \begin{minipage}{0.40\textwidth}
\includegraphics[clip, width=1\textwidth]{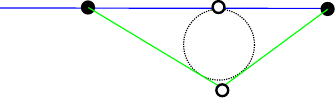}
\end{minipage} \vskip 0.05\textwidth  \begin{minipage}{0.40\textwidth}
\includegraphics[clip, width=0.8\textwidth]{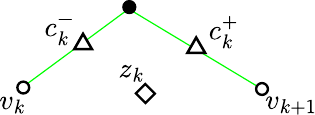}
\end{minipage} \begin{minipage}{0.45\textwidth}
\includegraphics[clip, width=\textwidth]{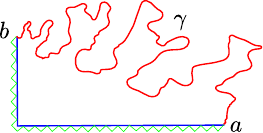}
\end{minipage}

\caption{TOP (Left): Slicing procedure creating some aligned boundary at level $y=0$. (Right): Extension using triangles viewed as tangential quadrilaterals. BOTTOM (Left): Boundary half-quad kite. (Right) $(m,\ell)$-corner with the its kite extension.}
\label{fig:kite-extension}
\end{figure}

The following lemma is key to understand the behaviour of FK-observables along straight wired arcs made of kites.  In the following lemma, $F$ denotes an $s$-holomorphic function associated with an FK Dobrushin Kadanoff-Ceva observable with \emph{wired} boundary conditions along the kites one level below $\cS \cap \mathcal{L}$, while $X$ is related to $F$ via \eqref{eq:X-from-F}.
\begin{lemma}
	In the previous setup, there exist $\phi_0>0$, only depending on for constants in \Unif\, such that for $z_k := (v_{k}^{\circ} v^\bullet v_{k+1}^{\circ})$ a boundary half-quad (which is a kite, crossed in from left to right) and $c_k^{\pm}\in z_k$ a boundary Kadanoff-Ceva corner one has for any $|\phi|\leq \tfrac{\pi}{4}+\phi_0.$
\begin{equation}\label{eq:boundary-comparison-fermion}
			\Re[F(z_k)\cdot e^{i\phi}] \geq 0 \quad \mathrm{ and } \quad \Re\Big[e^{i\phi} \Big( F(z_k)d\cS + i \overline{F(z_k)}d\cQ \Big)\Big] \asymp |X(c_k^\pm)|,
\end{equation}
\end{lemma}

\begin{proof}
Since the vector $\overrightarrow{v_k v_{k+1}}$ is purely real and positive, the boundary argument of $F(z_k)$ can be evaluated efficiently.  
Recall that along the boundary, we work with a \emph{wired} Kadanoff--Ceva fermion $X$, so the boundary half-quad value $F(z)$ satisfies
\begin{equation*}
	\Im\big[F^2(z)\,d\cS + i|F(z)|^2\,d\cQ\big]
	= H[F](v^\circ_{k+1}) - H[F](v^\circ_{k}) = 0,
\end{equation*}
where $d\cS = \cS(v^\circ_{k+1}) - \cS(v^\circ_{k}) \in \mathbb{R}^+$ and $d\cQ = \cQ(v^\circ_{k+1}) - \cQ(v^\circ_{k})$.  
Dividing by $d\cS$, one obtains
\begin{equation}\label{eq:boundary-argument-F}
	F^2(z) + i|F(z)|^2\,\frac{d\cQ}{d\cS} \in \mathbb{R}.
\end{equation}
Moreover, the boundary argument of Kadanoff-Ceva fermions along wired arcs is given by (see \cite[Lemma~5.3]{Che20})
\begin{equation*}
	\arg F(z_k) = \arg\Big(i\varsigma\big(\overline{\mathcal{X}}(c_k^+) - \overline{\mathcal{X}}(c_k^-)\big)\Big).
\end{equation*}
Hence, the left-hand side of \eqref{eq:boundary-argument-F} actually lies in $\mathbb{R}^+$. Since the extended grid satisfies a uniform \Uniffff\ property near the horizontal boundary (with constants depending only on those of the original grid $\cS$), the ratio $\frac{d\cQ}{d\cS}$ is uniformly bounded away from $1$.  
It follows that there exists a small $\phi_0 > 0$, depending only on the constants in \Uniffff\ (and thus on those of $\cS$), such that
\begin{equation}
	\arg F(z_k) \in \big]-\tfrac{\pi}{4} + 2\phi_0,\ \tfrac{\pi}{4} - 2\phi_0\big[.
\end{equation}
This establishes the left-hand side of \eqref{eq:boundary-comparison-fermion}.  

Furthermore, since
\begin{equation}
	F(z_k)\cdot \big(F(z_k)d\cS + i\,\overline{F(z)}\,d\cQ\big) \in \mathbb{R}^+,
\end{equation}
one has
\begin{equation} \label{eq:boundary-argument}
	\arg\big(F(z)\,d\cS + i\,\overline{F}(z)\,d\cQ\big)
	= \arg\big(\overline{F}(z)\big)
	\in \big]-\tfrac{\pi}{4} + 2\phi_0,\ \tfrac{\pi}{4} - 2\phi_0\big[.
\end{equation}
Combining the bounded-angle property of $\cS$, the reconstruction formula \eqref{eq:X-from-F}, 
and the observation \eqref{eq:boundary-argument}, we obtain \eqref{eq:boundary-comparison-fermion}.  
This completes the proof.
\end{proof}
Note that a similar construction exists for half-planes with other orientations. 
In particular, it is possible to construct a periodic boundary on $i\mathbb{H}$, 
the half-plane representing the region where $\Re[z] \geq 0$.

\begin{proof}
The proof follows the same general strategy as in \cite{DCHN,DMT21}, relying on the fact that the boundary contour integral nearly vanishes.  
Here, the additional layer of kites plays a key role in controlling the boundary argument of the FK-observable. Consider an $(m,\ell)$-corner domain $\Omega$ with Dobrushin boundary conditions, and let $\Omega'$ be the domain obtained by adding one horizontal layer of kites along both the bottom horizontal arc and the vertical arc-each extension being made along the wired boundary.  
We keep the same free arc $\gamma$. Let $X$ denote the FK-Dobrushin observable on $\Omega'$, with wired boundary conditions on the (kite-extended) arc $(ba)^\circ$ (oriented counterclockwise) and free boundary conditions on $\gamma = (ab)^\bullet$.  
Let $F$ be the corresponding complex $s$-holomorphic function defined via \eqref{eq:X-from-F}.  
For a positively oriented boundary contour, one has
\begin{equation}
	\Big|\oint_{\partial \Omega'} F\,d\cS + i\,\overline{F}\,d\cQ \Big|=  2,
\end{equation}
since interior boundary integrals vanish and branching occurs only at the boundary corners, where $X(a) = X(b) = 1$.  
We now decompose the boundary $\partial \Omega'$ into the wired arc $(ba)^\circ$ and the free arc $\gamma = (ab)^\bullet$.

\textbf{Contribution from $(ba)^\circ$.}  
Along the wired arc,
\begin{equation}
	\int_{(ba)^\circ} F\,d\cS + i\,\overline{F}\,d\cQ
	= \int_{\mathrm{vert}^\circ} F\,d\cS + i\,\overline{F}\,d\cQ
	+ \int_{\mathrm{hor}^\circ} F\,d\cS + i\,\overline{F}\,d\cQ,
\end{equation}
where $\mathrm{vert}^\circ$ and $\mathrm{hor}^\circ$ denote the vertical and horizontal parts of $(ba)^\circ$, respectively.  
Projecting this expression onto the line $e^{i\pi/4}\mathbb{R}$, the contribution of the vertical part is positive, since the boundary argument for kites in this vertical discretisation is in $]0;\frac{\pi}{2}[$ (and not a priori uniformly bounded away from $0$ or $\frac{\pi}{2}$ as the grid is not expected to be \Uniffff\ there).  Moreover, for a boundary kite $z_k \in \partial \Omega'$ on the horizontal segment between $m/3$ and $2m/3$, one has
\begin{equation}\label{eq:polynomial-bound}
	\Re\big[e^{-i\pi/4}\!\!\int_{v_k^\circ}^{v_{k+1}^\circ} F\,d\cS + i\,\overline{F}\,d\cQ\big]
	\asymp \phi^{0/1}_{\Omega'}(v_k^\bullet \longleftrightarrow (ab)^\bullet) \ge c\,m^{-1/2},
\end{equation}
using the exact one-arm exponent in the half-plane, which remains the same despite the  kite extension, enforced by the finite energy property. Summing over the $\asymp m/3$ primal vertices between $m/3$ and $2m/3$ yields the desired lower bound.

\textbf{Contribution from $(ab)^\bullet$.}  
For each $v^\bullet_k \sim v^\bullet_{k+1} \in (ab)^\bullet$ separated by $v_k^\circ \in \Omega$, one has
\begin{equation}
	|I[F](v^\bullet_{k+1}) - I[F](v^\bullet_{k})| \asymp \phi^{0/1}_{\Omega'}(v_k^\circ \longleftrightarrow (ba)^\circ).
\end{equation}
Hence there exists $c_3 > 0$, only depending on constants in \Unif\, such that
\begin{align}
	\sum_{v^\circ \in \gamma} \phi_{\Omega'}^{0/1}[v^\circ \longleftrightarrow (ab)^\circ]
	&\ge c_3 \sum_{z_k \in \gamma} |I[F](v^\bullet_{k+1}) - I[F](v^\bullet_{k})| \\
	&\ge \Big|\Re\!\Big[e^{-i\pi/4} \!\!\int_{(ab)^\bullet}\!\! F\,d\cS + i\,\overline{F}\,d\cQ \Big]\Big|\\
	&\ge \Big|\Re\!\Big[e^{-i\pi/4} \!\!\int_{(ba)^\circ}\!\! F\,d\cS + i\,\overline{F}\,d\cQ \Big]\Big| - 2,
\end{align}
which, combined with \eqref{eq:polynomial-bound}, establishes the desired bound for $\Omega'$.  
Finally, by monotonicity with respect to boundary conditions-bringing the wired arc closer (replacing the green arc with the red one)-the result extends to $\Omega$.\end{proof}
We are now in position to conclude for the proof of Proposition \ref{prop:M(r,R)} in a very similar way that the proof of Proposition 1.8 in \cite{DMT21} for $r=0$. In particular, the notations correspond exactly to those of used in \cite[Figure 8]{DMT21}, and we do not claim any novelty here.
\begin{proof}
Fix an $R$-centered domain and consider the largest
Euclidean ball $B$ centered at the origin and contained in $\mathcal{D}$. 
Let $x=(x_1,x_2)$ be a face of $\partial \mathcal{D} \cap \partial B$ and assume, without loss of generality, 
that $x$ is the wedge of $\{(u,v):u \geq v \geq 0\}$, where the ordering of $\mathbb{Z}^2$ is induced by 
the fundamental domains obtained as translates of the discrete horizontal and vertical half-planes constructed above. 
Let $y$ be the rightmost face of $\mathbb{N}^\star \times \{x_2+\delta R\}$ that is contained in $B$ and has the same type as $x$. 
Let $\tau$ be the translation mapping $x$ to $y$. 

Define a \emph{corner domain associated with} $\mathcal{D}$ as follows:  
set $a:=x-(8\delta^2 R,0)$, obtained by translating $x$ by shifts of fundamental domains so that $a$ has the same type as $x$;  
define $b:=\tau(a)$;  
let $\mathrm{Rect}$ denote the “rectangle” (its vertical and horizontal sides are taken as translates of $\partial \mathbb{H}$ and $\partial i\mathbb{H}$) with top-left corner $b$ and bottom corner $x$.  
Let $\Omega$ be the connected component of $a$ in
\[
(\mathcal{D}\cap \Lambda_{7R})\cap \tau(\mathcal{D}\cap \Lambda_{7R})\cap (\mathrm{Rect}\cup B^c).
\]
We will choose $\delta>0$ (unrelated to the $\delta$ scaling in previous sections) small enough.  

It is straightforward to see that the distance between $x$ and $y$ is at most $2\delta R|\mathcal{T}|$, 
while the distance along the horizontal line through $x$ between $x$ and the translated ball $\tau(B)$ 
is at most $4|\mathcal{T}|\delta^2 R$.  
For $\delta$ small and $R$ large, $\mathrm{Rect}$ is contained in $B\cap \tau(B)$ and thus part of $\partial \Omega$.  
We consider $(\Omega,a,b)$ as an $(m,\ell)$ Dobrushin corner with $m=\delta R$ and $\ell:=x_1-y_1$.  
Applying Lemma \ref{lem:lower-bound-crossing} gives
\begin{equation}\label{eq:contribution-arc}
	\sum\limits_{v^\circ\in (ab)}\phi_{\Omega}^{0/1}(z\longleftrightarrow (ba)^\circ) 
	\geq c_2\,(\delta R)^{1/2}.
\end{equation}  

We now decompose the boundary arc $(ab)$ into five sets:
\begin{itemize}
	\item $S_1$: faces of $\partial \mathcal{D}$,
	\item $S_2$: faces of $\tau(\partial \mathcal{D})$,
	\item $S_3$: faces of $\partial (\Lambda_{7R}\cap \tau(\Lambda_{7R}))$,
	\item $S_4$: horizontal faces between $a$ and $x$, denoted $[ax]$,
	\item $S_5$: horizontal faces between $b$ and $y$, denoted $[yb]$.
\end{itemize}

There are at most $16|\mathcal{T}|\delta^2 R$ vertices in $S_4\cup S_5$.  
To contribute to the sum in \eqref{eq:contribution-arc}, such vertices must realize a one-arm event in the half-plane 
starting from the bottom middle of the lower part of $\mathrm{Rect}$.  
For each face, this happens with probability at most $O((R\delta^2)^{-1/2})$, so by mixing we obtain
\begin{equation}\label{eq:contribution-arc}
	\sum\limits_{S_4\cup S_5}\phi_{\Omega}^{0/1}(z\longleftrightarrow (ba)^\circ) 
	\leq C\,R\,\delta^2\,(R\delta^2)^{-1/2},
\end{equation}
which is less than one quarter of the right-hand side of \eqref{eq:contribution-arc} when $\delta$ is small enough.  

For $v^\circ\in S_3$, the mixing property gives
\[
	\phi_{\Omega}^{0/1}(v^\circ\longleftrightarrow (ba)^\circ) 
	\leq C' \,\pi_{1}^{+}(\delta R)\, \phi_{\Omega}^{0/1}((ba)\longleftrightarrow \partial \Lambda_{6R}).
\]
We now bound the last term on the right-hand side.  
Here, the cluster must cross a free corridor of width $\delta R$ and length $6R$.  
As in the original proof, this occurs with probability at most $e^{-c\delta^{-1}}$ 
for some universal constant $c>0$ depending only on the constants of Theorem \ref{thm:RSW-FK}.  
Thus, for $\delta$ small, the one-arm exponent implies that this contribution is also less than one quarter 
of the right-hand side of \eqref{eq:contribution-arc}.  

Consequently,
\begin{equation}
	\sum\limits_{v^\circ\in S_1\cup S_2}\phi_{\Omega}^{0/1}(v^\circ \longleftrightarrow (ba)^\circ) 
	\geq c(\delta)\,R^{1/2}.
\end{equation}

Finally, the same argument as in the original proof of \cite{DMT21} 
converts this large $\phi_{\Omega}^{0/1}$ expectation into a large $\phi_{\mathcal{D}}^{0}$ expectation on boundary vertices connected to $\Lambda_R$.  
This uses only RSW-type estimates and basic properties of the random-cluster model, 
which apply here without change.\end{proof} 

\begin{proof}[Proof of Theorem \ref{thm:SuperStrongRSW}]
Once Proposition \ref{prop:M(r,R)} is established, all the renormalization arguments of \cite{DMT21} apply almost verbatim in our setting, allowing us to conclude in the same way.
\end{proof}

\section{Concluding argument using CLE(16/3) convergence}

We now use the CLE convergence to conclude about the scaling limit of spin correlations. Let us state this precisely. In what follows, $d_\mathcal{X}$ denotes the standard metric on the space of loop collections, recalled in \cite[Section 2.7]{BenHon19}. 

\begin{theo}[Kemppainen-Smirnov \cite{KemSmi2}]\label{thm:CLE-convergence}
 Fix a Jordan domain $\Omega$ approximated in the Hausdorff sense by $\Omega^\delta\subset \cS^\delta$. Then for the topology of convergence $d_\mathcal{X}$, the collections of loops converges to $\mathrm{CLE_{\Omega}(16/3)}$.
\end{theo}

This theorem combines three ingredients: the convergence of FK-martingales observables to their continuous counterpart and of the FK-interfaces to SLE$(16/3)$ from \cite{Che20}, the Russo-Seymour-Welsh estimate in terms of extremal length (Theorem~\ref{thm:SuperStrongRSW}), all plugged into the analysis of \cite{KemSmi2}. As already stated in the introduction, the convergence to CLE loop ensembles is done here in the natural topology for convergence of loop ensembles, while one could wonder if the same statement holds for the Schramm-Smirnov quad topology of \cite{schramm2011scaling}. This statement is widely believed to hold but might currently be missing some effective complete reference. Still, for very simple connectivity events that we use, the former is sufficient.

 We now turn to the proof of Theorem~\ref{thm:Correlation-function}. Beforehand, we recall basic connections between the Ising and FK--Ising models. Let $\Omega^\delta$ be a simply connected discrete domain with \emph{wired} boundary conditions for Ising, and let $\phi_{\Omega^\delta}^{1}$ be the corresponding FK measure. By the Edwards--Sokal coupling, for any faces $a^\delta,b^\delta,c^\delta,d^\delta\in \Omega^\delta$ one has
\[
\frac{\mathbb{E}_{\Omega^\delta}^{(\mathrm{w})}[\sigma_{a^\delta}\sigma_{b^\delta}]}{\mathbb{E}_{\Omega^\delta}^{(\mathrm{w})}[\sigma_{c^\delta}\sigma_{d^\delta}]}
=
\frac{\phi_{\Omega^\delta}^{1}[\,a^\delta \longleftrightarrow b^\delta\,]}{\phi_{\Omega^\delta}^{1}[\,c^\delta \longleftrightarrow d^\delta\,]}.
\]
This identity will be our main tool for proving convergence and universality of spin correlations. Throughout, assume $\Omega^\delta \to \Omega$ in the Hasdorff sense. As a first intermediate step, we relate point-to-point connection probabilities to connections of macroscopic loops in the CLE ensemble. Heuristically, as $\delta\to0$, for two points to lie in the same FK cluster, each must first connect to a small macroscopic circle, after which a cluster connects the two circles. We model this as (almost) independent one-arm events around each point, followed by a circle-to-circle connection expressed via CLE estimates. The key object here is the Incipient Infinite Cluster (I.I.C.\,) for the FK-Ising model, originally developed for percolation \cite{SBRN76,kesten1986incipient}; see also \cite[Appendix~A]{oulamara2022random}. Thanks to Theorem~\ref{thm:SuperStrongRSW}, we may invoke \emph{quasimultiplicativity} of arm events as in \cite[Proposition~6.3]{DMT21}, also recalled below.
 Fix a face $f_0$ and $R>0$ and $B$ be an event only depending on finitely many edges. For $R$ large enough, and $N\geq R$, define 
 \begin{equation*}
 	\mathbb{P}_{N}(B):= \phi^{0}_{\Lambda_N}(B| \partial \Lambda_R(f_0) \longleftrightarrow \partial \Lambda_N(f_0))
 \end{equation*}
 Using the technology recalled in \cite[Proposition A.2.1]{oulamara2022random}, it is possible to construct a true limiting measure (using Kolmogorov extension), denoted $ \Phi^{\Lambda_R(f_0)}_{\textrm{IIC}}$ such that 
 \begin{equation*}
 \Phi^{\Lambda_R(f_0)}_{\textrm{IIC}}(B)=\lim\limits_{N\to \infty}\mathbb{P}_{N}(B).
 \end{equation*}
Informally speaking, this measure is the FK measure conditioning the boundary of the box $\Lambda_R(f_0) $ to be connected to infinity. Moreover, this limiting convergence is polynomial, as recalled in \cite[Proposition A.2.2]{oulamara2022random} (for a more complicated setup of $3$ arm events), reconstructing arguments of \cite{garban2013pivotal}.  For $n\geq R$, let $\mathcal{F}_n$ be the set of events only depending on edges of $\Lambda_n$. Then there exist $c>0$, only depending on the geometry of $\cS$ (and therefore on the associated RSW bounds obtained from Theorem \ref{thm:RSW-FK}) such that 
\begin{equation*}
	\sup\limits_{B\in \mathcal{F}_n} \frac{|\Phi^{\Lambda_R(f_0)}_{\textrm{IIC}}(B)- \mathbb{P}_{N}(B)|}{\Phi^{\Lambda_R(f_0)}_{\textrm{IIC}}(B)} \leq \frac{1}{c}\cdot \Big(\frac{n}{N}\Big)^c
\end{equation*}
We keep introducing additional notations.  For a point $v \in \mathbb{C}$ and a periodic grid $\cS^\delta$, let $\mathcal{C}_\varepsilon(v)$ denote a discrete circle of radius $\varepsilon$ centred at $v$, defined as the boundary of the discrete ball $\mathcal{B}_\varepsilon(v)$. We assume for simplicity that the two  discrete circles are of the same radius, centred at different point, are obtained by translation. For $a^\delta \in \cS^\delta$, $\delta\leq r_1\leq r_2 $  denote by 
\begin{equation}\label{eq:def-1-arm}
	\pi_{1}(a^\delta,r_1,r_2):= \phi_{\cS^\delta}\Big[ \mathcal{C}_{r_1}(a^\delta) \leftrightarrow  \mathcal{C}_{r_2}(a^\delta) \Big].
\end{equation}
Note that separation of arm events comes from Theorem \ref{thm:SuperStrongRSW} (see \cite[Proposition 6.3]{DMT21}) and reads here (for some constant $\asymp$ only depending on \Unif\,) as
\begin{equation}\label{eq:mixing}
	\pi_{1}(a^\delta,r_1,r_3) \asymp \pi_{1}(a^\delta,r_1,r_2) \cdot \pi_{1}(a^\delta,r_2,r_3)
\end{equation}
the probability of a one-arm event in the full-plane. One can also define a four-arm event, where the discrete circles $\mathcal{C}_{r_1}(a^\delta)$ and $\mathcal{C}_{r_2}(a^\delta)$ are connected by $2$ disjoint clusters (separated by two disjoint dual clusters), whose probability is denoted by
\begin{equation}\label{eq:def-4-arm}
	\pi_{4}(a^\delta,r_1,r_2):= \phi_{\cS^\delta}\Big[ \mathcal{C}_{r_1}(a^\delta) \overset{ \textrm{(2)}}{\longleftrightarrow}  \mathcal{C}_{r_2}(a^\delta) \Big].
\end{equation}

By rescaling the lattice, one can define the measure $\mathbb{P}^\delta_{\mathrm{IIC}}$ on $\cS^\delta$. We now establish the following lemma, which serves as the key ingredient for expressing FK cluster connection probabilities in terms of CLE events and I.I.C.\ connection probabilities. The author thanks Emile Av\'erous and Tiancheng He for enlightening discussions on the I.I.C.\ measure that allowed to prove the statement below.
\begin{lemma}\label{lem:IIC-writing}
	Fix $\varepsilon>0$ and two faces $a^\delta,b^\delta \in \Omega^\delta$, approximating respectively the points $a,b\in \Omega$, which lie at a definite distance $r$ from each other and from the boundary. For $\delta \ll \varepsilon \leq \sqrt{\varepsilon}\leq r/100 $, there exist $c>0$, only depending on constants in \Unif\,, such that one has, uniformly in $\delta$ small enough
\begin{equation*}
	\phi_{\Omega^\delta}^{1}\Big[ a^\delta \leftrightarrow b^\delta \Big] = \phi_{\Omega^\delta}^{1}\Big[ \mathcal{C}_{\varepsilon}(a^\delta) \leftrightarrow \mathcal{C}_{\varepsilon}(b^\delta) \Big] \Phi^{\mathcal{C}_{\varepsilon}(a^\delta)}_{\mathrm{IIC}}\Big[ a^\delta \leftrightarrow  \infty  \Big] \cdot \Phi^{\mathcal{C}_{\varepsilon}(b^\delta)}_{\mathrm{IIC}}\Big[ b^\delta \leftrightarrow  \infty  \Big],
\end{equation*}	
up to a multiplicative error $(1+O(\varepsilon^c))(1+O((\sqrt{\varepsilon}/r)^c)$. Note that the event $\{ a^\delta \leftrightarrow  \infty \}  $ depends on infinitely many edges, but can still be approximated in a quantitative way by events depending on finitely many edges.
	\end{lemma}
\begin{proof}
	We argue by a multiscale analysis, using the polynomial convergence in the construction of the I.I.C.\ measure. Before the details, we fix some auxiliary notation. Set
\[
M_{\varepsilon}(a^\delta,b^\delta):=\mathcal{B}_{\sqrt{\varepsilon}}(a^\delta)^{c}\cap\mathcal{B}_{\sqrt{\varepsilon}}(b^\delta)^{c}.
\]
For a cluster configuration $\omega$ on $\Omega^\delta$, let its restriction to $M_{\varepsilon}(a^\delta,b^\delta)$ be
$\omega^{(\varepsilon)}:=\omega\!\restriction_{M_{\varepsilon}(a^\delta,b^\delta)}$.
Define the event
\[
\mathrm{Uni}_{\varepsilon}(a^\delta,b^\delta)
:=\big\{\mathcal{C}_{\sqrt{\varepsilon}}(a^\delta)\overset{!}{\longleftrightarrow}\mathcal{C}_{\sqrt{\varepsilon}}(b^\delta)\big\},
\]
that there exists a \emph{unique} cluster $\Gamma_\varepsilon^\delta$ connecting the two discrete circles of radius $\sqrt{\varepsilon}$ centred at $a^\delta$ and $b^\delta$. Note that the event $\mathrm{Uni}_{\varepsilon}(a^\delta,b^\delta)$ is measurable with respect to the edges in $M_{\varepsilon}(a^\delta,b^\delta)$. We then decompose
\begin{align*}
\phi_{\Omega^\delta}^{1}\!\big[a^\delta \leftrightarrow b^\delta\big]
&=\phi_{\Omega^\delta}^{1}\!\big[a^\delta \leftrightarrow b^\delta \cap \,\mathrm{Uni}_{\varepsilon}(a^\delta,b^\delta)\big]
 +\phi_{\Omega^\delta}^{1}\!\big[a^\delta \leftrightarrow b^\delta \cap \,\mathrm{Uni}_{\varepsilon}^{\frak c}(a^\delta,b^\delta)\big]\\
&=\sum_{\omega^{(\varepsilon)}\in \mathrm{Uni}_{\varepsilon}}
\phi_{\Omega^\delta}^{1}\!\big[\omega^{(\varepsilon)}\big]\,
\phi_{\Omega^\delta}^{1}\!\big[a^\delta \leftrightarrow b^\delta \,\big|\, \omega^{(\varepsilon)}\big]
+\phi_{\Omega^\delta}^{1}\!\big[a^\delta \leftrightarrow b^\delta \cap \,\mathrm{Uni}_{\varepsilon}^{\frak c}(a^\delta,b^\delta)\big],
\end{align*}
where $A^{\frak c}$ denotes the complement of the event $A$.

\textbf{Step 0: Rough estimate of $\phi_{\Omega^\delta}^{1}[\,a^\delta \leftrightarrow b^\delta\,]$.}
We begin with a \emph{rough} estimate of $\phi_{\Omega^\delta}^{1}[\,a^\delta \leftrightarrow b^\delta\,]$, up to constants $\asymp$ depending only on \Unif. By the quasimultiplicativity property \eqref{eq:mixing},
\begin{equation}\label{eq:rough-estimate-domain}
\phi_{\Omega^\delta}^{1}[\,a^\delta \leftrightarrow b^\delta\,]
\asymp
\pi_{1}(a^\delta,\delta,\sqrt{\varepsilon})\,
\pi_{1}(a^\delta,\sqrt{\varepsilon},r)\,
\pi_{1}(b^\delta,\delta,\sqrt{\varepsilon})\,
\pi_{1}(b^\delta,\sqrt{\varepsilon},r).
\end{equation}

\textbf{Step 1: Precise estimate of $\phi_{\Omega^\delta}^{1}[\,a^\delta \leftrightarrow b^\delta \cap \,\mathrm{Uni}_{\varepsilon}(a^\delta,b^\delta)\,]$.}
On $\omega^{(\varepsilon)}\in\mathrm{Uni}_{\varepsilon}(a^\delta,b^\delta)$, the uniqueness of the cluster connecting the annuli boundaries $\mathcal{C}_{\sqrt{\varepsilon}}(a^\delta)$ and $\mathcal{C}_{\sqrt{\varepsilon}}(b^\delta)$ implies
\[
\phi_{\Omega^\delta}^{1}\!\big[a^\delta \leftrightarrow b^\delta \,\big|\, \omega^{(\varepsilon)}\big]
=
\phi_{\Omega^\delta}^{1}\!\big[a^\delta \leftrightarrow \Gamma_\varepsilon^\delta \,\big|\, \omega^{(\varepsilon)}\big]\,
\phi_{\Omega^\delta}^{1}\!\big[b^\delta \leftrightarrow \Gamma_\varepsilon^\delta \,\big|\, \omega^{(\varepsilon)}\big],
\]
since the connections from $a^\delta$ and $b^\delta$ to $\Gamma_\varepsilon^\delta$ are conditionally independent inside $\mathcal{B}_{\sqrt{\varepsilon}}(a^\delta)$ and $\mathcal{B}_{\sqrt{\varepsilon}}(b^\delta)$. Moreover,
\[
\phi_{\Omega^\delta}^{1}\!\big[a^\delta \leftrightarrow \Gamma_\varepsilon^\delta \,\big|\, \omega^{(\varepsilon)}\big]
=
\phi_{\Omega^\delta}^{1}\!\big[\mathcal{C}_{\varepsilon}(a^\delta)\leftrightarrow \Gamma_\varepsilon^\delta \,\big|\, \omega^{(\varepsilon)}\big]\,
\phi_{\Omega^\delta}^{1}\!\big[a^\delta \leftrightarrow \Gamma_\varepsilon^\delta \,\big|\, \omega^{(\varepsilon)},\, \mathcal{C}_{\varepsilon}(a^\delta)\leftrightarrow \Gamma_\varepsilon^\delta\big].
\]
Here polynomial convergence to the I.I.C.\ and decorrelation across scales enter: choosing scales with $\delta \ll \varepsilon \ll \sqrt{\varepsilon}$, there exists $c_1>0$, depending only on \Unif, such that for $\omega^{(\varepsilon)}\in\mathrm{Uni}_{\varepsilon}$,
\[
\phi_{\Omega^\delta}^{1}\!\big[a^\delta \leftrightarrow \Gamma_\varepsilon^\delta \,\big|\, \omega^{(\varepsilon)},\, \mathcal{C}_{\varepsilon}(a^\delta)\leftrightarrow \Gamma_\varepsilon^\delta\big]
=
 \Phi^{\mathcal{C}_{\varepsilon}(a^\delta)}_{\mathrm{IIC}}\Big[ a^\delta \leftrightarrow  \infty  \Big] \,
\Big(1+O\big(\varepsilon^{c_1}\big)\Big).
\]
Note that formally speaking, in order to rigorously write this last equation, one should make a slightly more precise analysis on the convergence to the I.I.C.\,, decomposing the event the event $a^\delta \leftrightarrow \Gamma_\varepsilon^\delta$ into a \emph{local} version, where $\Gamma_\varepsilon^\delta \leftrightarrow \mathcal{C}_{\varepsilon}(a^\delta)$ inside an intermediate scale, say $\varepsilon^{2/3}$, and its complementary, which creates a four arm events between the scales $\varepsilon $ and $\varepsilon^{2/3}$ and has therefore some power in $\varepsilon$ smaller probability, preserving the announced bound.
Similarly (by translation invariance on $\cS^\delta$),
\[
\phi_{\Omega^\delta}^{1}\!\big[b^\delta \leftrightarrow \Gamma_\varepsilon^\delta \,\big|\, \omega^{(\varepsilon)},\, \mathcal{C}_{\varepsilon}(b^\delta)\leftrightarrow \Gamma_\varepsilon^\delta\big]
=
 \Phi^{\mathcal{C}_{\varepsilon}(b^\delta)}_{\mathrm{IIC}}\Big[ b^\delta \leftrightarrow  \infty  \Big]\,
\Big(1+O\big(\varepsilon^{c_1}\big)\Big).
\]
Putting together all the above statements, one has 
\begin{multline*}
	\sum\limits_{\omega^{(\varepsilon)}\in \textrm{Uni}_{\varepsilon}} \phi_{\Omega^\delta}\Big[\omega^{(\varepsilon)}\Big]\phi_{\Omega^\delta}\Big[a^\delta \leftrightarrow b^\delta \:| \:\omega^{(\varepsilon)}\Big]=\Phi^{\mathcal{C}_{\varepsilon}(a^\delta)}_{\mathrm{IIC}}\Big[a^\delta \leftrightarrow  \mathcal{C}_{\varepsilon}(a^\delta)  \Big] \Phi^{\mathcal{C}_{\varepsilon}(b^\delta)}_{\mathrm{IIC}}\Big[b^\delta \leftrightarrow  \mathcal{C}_{\varepsilon}(b^\delta)  \Big] \times \cdots \\
\cdots \times \sum\limits_{\omega^{(\varepsilon)}\in \textrm{Uni}_{\varepsilon}} \phi_{\Omega^\delta}\Big[\omega^{(\varepsilon)}\Big]\phi_{\Omega^\delta}\Big[ \mathcal{C}_{\varepsilon}(a^\delta) \leftrightarrow \Gamma_\varepsilon^\delta \: | \:\omega^{(\varepsilon)}\Big]\phi_{\Omega^\delta}\Big[ \mathcal{C}_{\varepsilon}(b^\delta) \leftrightarrow \Gamma_\varepsilon^\delta \: | \:\omega^{(\varepsilon)}\Big],
\end{multline*}
up to a multiplicative pre-factor $(1+O(\varepsilon^{c_1}))$. Moreover, using again conditional independence inside the respective balls $\mathcal{B}_{\sqrt{\varepsilon}}(a^\delta)$ and $\mathcal{B}_{\sqrt{\varepsilon}}(b^\delta)$ one gets that 
\begin{multline*}
\phi_{\Omega^\delta}^{1}\Big[ a^\delta \leftrightarrow b^\delta \cap \textrm{Uni}_{\varepsilon}(a^\delta,b^\delta) \Big]= \Phi^{\mathcal{C}_{\varepsilon}(a^\delta)}_{\mathrm{IIC}}\Big[ a^\delta \leftrightarrow  \infty  \Big]  \Phi^{\mathcal{C}_{\varepsilon}(b^\delta)}_{\mathrm{IIC}}\Big[ b^\delta \leftrightarrow  \infty  \Big] \times \cdots\\ 
\cdots \times \sum\limits_{\omega^{(\varepsilon)}\in \textrm{Uni}_{\varepsilon}} \phi_{\Omega^\delta}\Big[\omega^{(\varepsilon)}\Big]\phi_{\Omega^\delta}\Big[ \mathcal{C}_{\varepsilon}(a^\delta) \leftrightarrow \Gamma_\varepsilon^\delta \: | \:\omega^{(\varepsilon)}\Big]\phi_{\Omega^\delta}\Big[ \mathcal{C}_{\varepsilon}(b^\delta) \leftrightarrow \Gamma_\varepsilon^\delta \: | \:\omega^{(\varepsilon)}\Big]\\
=\Phi^{\mathcal{C}_{\varepsilon}(a^\delta)}_{\mathrm{IIC}}\Big[ a^\delta \leftrightarrow  \infty  \Big]  \Phi^{\mathcal{C}_{\varepsilon}(b^\delta)}_{\mathrm{IIC}}\Big[ b^\delta \leftrightarrow  \infty  \Big] \phi_{\Omega^\delta}\Big[ \mathcal{C}_{\varepsilon}(a^\delta) \leftrightarrow \mathcal{C}_{\varepsilon}(b^\delta) \cap \textrm{Uni}_{\varepsilon}(a^\delta,b^\delta)  \Big],
\end{multline*}
again up to a multiplicative pre-factor $(1+O(\varepsilon^{c_1}))$.

\textbf{Step 2: Rough estimate of $\phi_{\Omega^\delta}^{1}[\,a^\delta \leftrightarrow b^\delta\cap \,\mathrm{Uni}^{\frak c}_{\varepsilon}(a^\delta,b^\delta)\,]$.}
On $\mathrm{Uni}^{\frak c}_{\varepsilon}(a^\delta,b^\delta)$ there must be, between scales $\sqrt{\varepsilon}$ and $r$, at least two four-arm events, one around $a^\delta$ and one around $b^\delta$. By the mixing property, there is a constant (denoted $\lesssim$, depending only on \Unif) such that
\begin{equation}\label{eq:non-unic-rough-estimate}
\phi_{\Omega^\delta}^{1}\!\big[\,a^\delta \leftrightarrow b^\delta\cap \,\mathrm{Uni}^{\frak c}_{\varepsilon}(a^\delta,b^\delta)\,\big]
\;\lesssim\;
\pi_{1}(a^\delta,\delta,\sqrt{\varepsilon})\,
\pi_{4}(a^\delta,\sqrt{\varepsilon},r)\,
\pi_{1}(b^\delta,\delta,\sqrt{\varepsilon})\,
\pi_{4}(b^\delta,\sqrt{\varepsilon},r).
\end{equation}
Using the arm-separation property, four-arm events are polynomially rarer than one-arm events: there exists $c_2>0$ such that
\begin{equation}\label{eq:polynomial-improvement-1arm-4arm}
\pi_{4}(a^\delta,\sqrt{\varepsilon},r)\;\le\; c_2\Big(\frac{\sqrt{\varepsilon}}{r}\Big)^{c_2}\,\pi_{1}(a^\delta,\sqrt{\varepsilon},r),
\end{equation}
and similarly for $b^\delta$. Combining \eqref{eq:rough-estimate-domain}, \eqref{eq:non-unic-rough-estimate}, and \eqref{eq:polynomial-improvement-1arm-4arm} yields
\[
\phi_{\Omega^\delta}^{1}\!\big[\,a^\delta \leftrightarrow b^\delta\,\big]
=
\Phi^{\mathcal{C}_{\varepsilon}(a^\delta)}_{\mathrm{IIC}}\!\big[a^\delta \leftrightarrow \infty\big]\,
\Phi^{\mathcal{C}_{\varepsilon}(b^\delta)}_{\mathrm{IIC}}\!\big[b^\delta \leftrightarrow \infty\big]\,
\phi_{\Omega^\delta}^{1}\!\big[\,\mathcal{C}_{\varepsilon}(a^\delta) \leftrightarrow \mathcal{C}_{\varepsilon}(b^\delta)\cap \,
\mathrm{Uni}_{\varepsilon}(a^\delta,b^\delta)\big],
\]
up to a prefactor $(1+O(\varepsilon^{c_1}))(1+O((\sqrt{\varepsilon}/r)^{c_2}))$ (for $c_1,c_2>0$ depending only on \Unif). Moreover, adapting the previous arguments,
\[
\phi_{\Omega^\delta}^{1}\!\big[\,\mathcal{C}_{\varepsilon}(a^\delta) \leftrightarrow \mathcal{C}_{\varepsilon}(b^\delta)\cap \,
\mathrm{Uni}_{\varepsilon}(a^\delta,b^\delta)\big]
=
\phi_{\Omega^\delta}^{1}\!\big[\,\mathcal{C}_{\varepsilon}(a^\delta) \leftrightarrow \mathcal{C}_{\varepsilon}(b^\delta)\big]\,
\big(1+O(\varepsilon^{c_3})\big)
\]
for some $c_3>0$. Choosing $c>0$ small and absorbing constants, we obtain the desired bound with overall prefactor $(1+O(\varepsilon^{c}))(1+O((\sqrt{\varepsilon}/r)^{c}))$.
	\begin{figure}
\hskip -0.01\textwidth\begin{minipage}{0.49\textwidth}
\includegraphics[clip, width=1\textwidth]{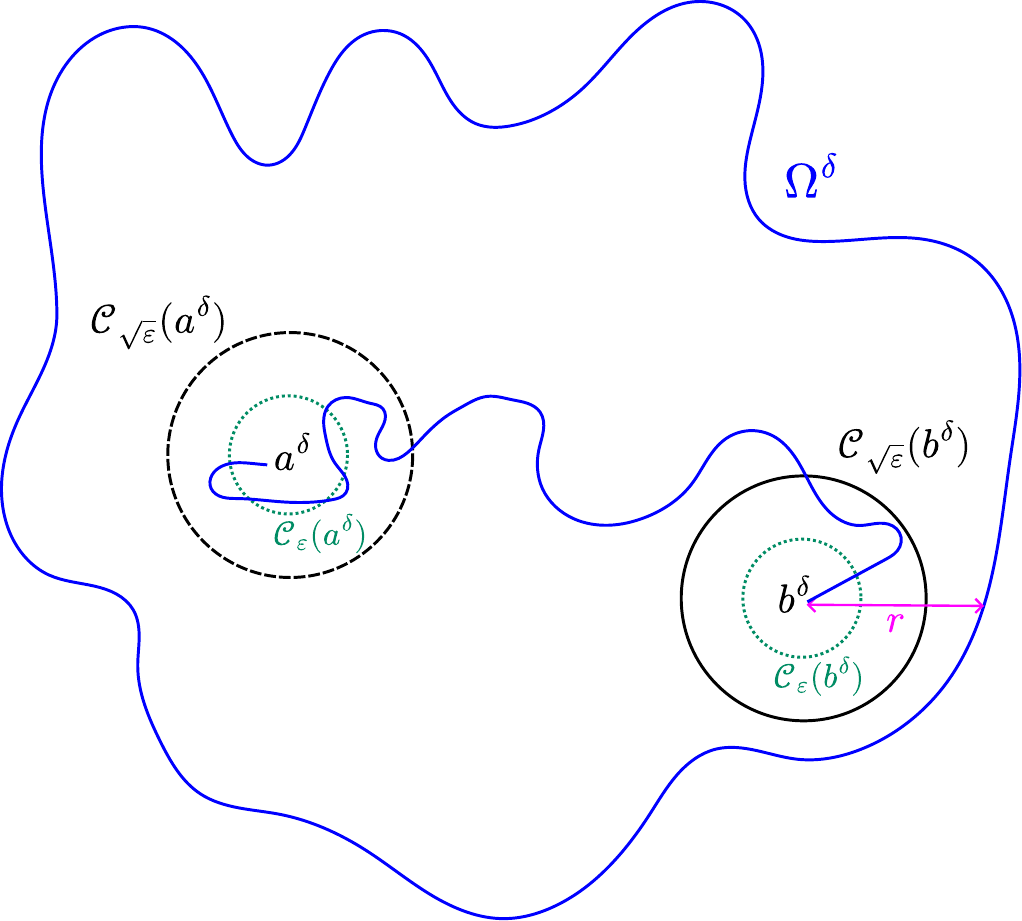}
\end{minipage}\hskip 0.01\textwidth \begin{minipage}{0.49\textwidth}
\includegraphics[clip, width=0.8\textwidth]{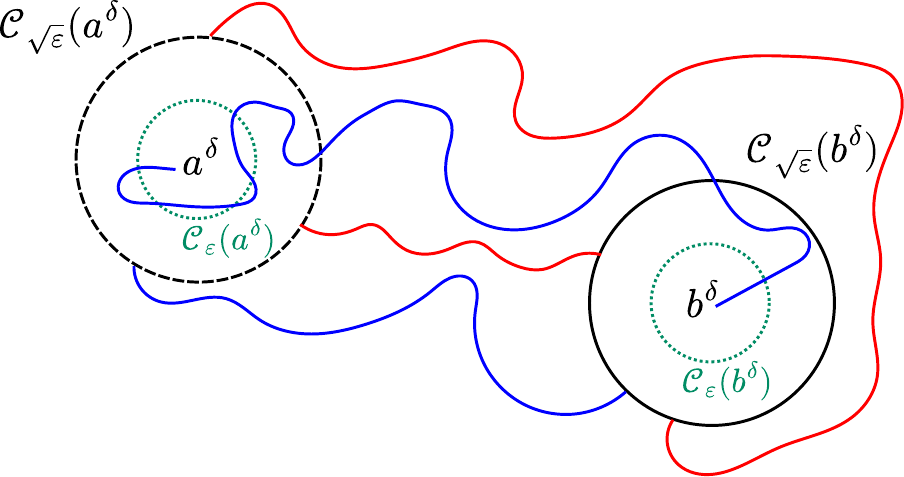}
\end{minipage}
\caption{(Left) Notations used within the proof. (Right) Configuration where $\textrm{Uni}_{\varepsilon}(a^\delta,b^\delta)$ inducing two four-arm events (centred respectively at $a^\delta$ and $b^\delta$) between the distances $\sqrt{\varepsilon}$ and $r$. This events are polynomially more unlikely compared with one-arm events at the same scales.}
\end{figure}
\end{proof}

\begin{theo}\label{thm:convergence-ratios}
	Fix $\varepsilon>0$ and four faces $a^\delta,b^\delta,c^\delta,d^\delta \cS^\delta$ \emph{of the same type}, approximating respectively the points $a,b,c,d$, at a definite distance $r>10\varepsilon$ from each other and from the boundary.  The uniformly on compacts of $\Omega$ and the distance between the points $a,b,c,d$, one has  
	\begin{equation*}
		\frac{\phi^{1}_{\Omega^\delta}[a^\delta \longleftrightarrow b^\delta]}{\phi^{1}_{\Omega^\delta}[c^\delta \longleftrightarrow d^\delta]} \underset{\delta \to 0}{\longrightarrow} \frac{\langle \sigma_a \sigma_b \rangle_{\Omega}^{(\mathrm{w})}}{\langle \sigma_c \sigma_d \rangle_{\Omega}^{(\mathrm{w})}},
	\end{equation*}
where the correlation function is defined in Section \ref{sub:correlation} and is the same as the square lattice.
\end{theo}
\begin{proof}
This proof combines Lemma~\ref{lem:lower-bound-crossing} with a simple but crucial observation: the known convergence on the square lattice, obtained by completely different methods. Fix $\varepsilon>0$ and scales $\delta \ll \varepsilon \le \sqrt{\varepsilon} \le r/100$. Applying Lemma~\ref{lem:IIC-writing}, we can write, up to an error of order $(1+O(\varepsilon^{c}))(1+O((\sqrt{\varepsilon}/r)^{c}))$,
\begin{equation*}
  \frac{\phi_{\Omega^\delta}^{1}[\,a^\delta \leftrightarrow b^\delta\,]}{\phi_{\Omega^\delta}^{1}[\,c^\delta \leftrightarrow d^\delta\,]}
  =
  \frac{\phi_{\Omega^\delta}^{1}[\,\mathcal{C}_{\varepsilon}(a^\delta) \leftrightarrow \mathcal{C}_{\varepsilon}(b^\delta)\,]
  \,\Phi^{\mathcal{C}_{\varepsilon}(a^\delta)}_{\mathrm{IIC}}[\,a^\delta \leftrightarrow \infty\,]
  \,\Phi^{\mathcal{C}_{\varepsilon}(b^\delta)}_{\mathrm{IIC}}[\,b^\delta \leftrightarrow \infty\,]}
  {\phi_{\Omega^\delta}^{1}[\,\mathcal{C}_{\varepsilon}(c^\delta) \leftrightarrow \mathcal{C}_{\varepsilon}(d^\delta)\,]
  \,\Phi^{\mathcal{C}_{\varepsilon}(c^\delta)}_{\mathrm{IIC}}[\,c^\delta \leftrightarrow \infty\,]
  \,\Phi^{\mathcal{C}_{\varepsilon}(d^\delta)}_{\mathrm{IIC}}[\,d^\delta \leftrightarrow \infty\,]}.
\end{equation*}
Since the faces $a^\delta,b^\delta,c^\delta,d^\delta$ are of the same type, the uniqueness of the I.I.C.\ measure and translation invariance of $\cS^\delta$ imply that the $\mathbb{P}_{\mathrm{IIC}}$ terms cancel, giving
\begin{equation}\label{eq:convergence-eq1}
  \frac{\phi_{\Omega^\delta}^{1}[\,a^\delta \leftrightarrow b^\delta\,]}{\phi_{\Omega^\delta}^{1}[\,c^\delta \leftrightarrow d^\delta\,]}
  =
  \frac{\phi_{\Omega^\delta}^{1}[\,\mathcal{C}_{\varepsilon}(a^\delta) \leftrightarrow \mathcal{C}_{\varepsilon}(b^\delta)\,]}
  {\phi_{\Omega^\delta}^{1}[\,\mathcal{C}_{\varepsilon}(c^\delta) \leftrightarrow \mathcal{C}_{\varepsilon}(d^\delta)\,]}
  \big((1+O(\varepsilon^{c}))(1+O((\sqrt{\varepsilon}/r)^{c}))\big).
\end{equation}
Taking $\delta \to 0$, and using CLE convergence (Theorem~\ref{thm:CLE-convergence}) together with the fact that discrete circles $\mathcal{C}_{\varepsilon}(\cdot)$ are small but macroscopic, we obtain
\begin{equation*}
	\limsup_{\delta \to 0}
	\frac{\phi_{\Omega^\delta}^{1}[\,a^\delta \leftrightarrow b^\delta\,]}
	{\phi_{\Omega^\delta}^{1}[\,c^\delta \leftrightarrow d^\delta\,]}
	\le
	\frac{\mathbb{P}^{\mathrm{CLE}}_{\Omega}[\,\mathcal{C}_{\varepsilon}(a) \leftrightarrow \mathcal{C}_{\varepsilon}(b)\,]}
	{\mathbb{P}^{\mathrm{CLE}}_{\Omega}[\,\mathcal{C}_{\varepsilon}(c) \leftrightarrow \mathcal{C}_{\varepsilon}(d)\,]}
	\big((1+O(\varepsilon^{c}))(1+O((\sqrt{\varepsilon}/r)^{c}))\big).
\end{equation*}
The remaining challenge is to estimate CLE connection probabilities in the small-$\varepsilon$ regime. To circumvent this difficulty, we apply \eqref{eq:convergence-eq1} in the special case where $\cS^\delta$ is the critical square lattice. Sending $\delta\to0$, it follows from \cite[Theorem~1.2]{CHI} that
\begin{equation*}
	\lim_{\delta \to 0}
	\frac{\phi_{\Omega^\delta}^{1}[\,a^\delta \leftrightarrow b^\delta\,]}
	{\phi_{\Omega^\delta}^{1}[\,c^\delta \leftrightarrow d^\delta\,]}
	=
	\frac{\langle \sigma_a \sigma_b \rangle_{\Omega}^{(\mathrm{w})}}
	{\langle \sigma_c \sigma_d \rangle_{\Omega}^{(\mathrm{w})}},
\end{equation*}
while CLE convergence gives
\begin{equation*}
\lim_{\delta \to 0}
\frac{\phi_{\Omega^\delta}^{1}[\,\mathcal{C}_{\varepsilon}(a^\delta) \leftrightarrow \mathcal{C}_{\varepsilon}(b^\delta)\,]}
{\phi_{\Omega^\delta}^{1}[\,\mathcal{C}_{\varepsilon}(c^\delta) \leftrightarrow \mathcal{C}_{\varepsilon}(d^\delta)\,]}
=
\frac{\mathbb{P}^{\mathrm{CLE}}_{\Omega}[\,\mathcal{C}_{\varepsilon}(a) \leftrightarrow \mathcal{C}_{\varepsilon}(b)\,]}
{\mathbb{P}^{\mathrm{CLE}}_{\Omega}[\,\mathcal{C}_{\varepsilon}(c) \leftrightarrow \mathcal{C}_{\varepsilon}(d)\,]}.
\end{equation*}
Note that formally, one should additionally check that in continuum, a.s.\, any CLE loop surrounding one of the continuous balls of size $\varepsilon$ centred at $a$ or $b$ doesn't touch the boundary the ball. 

Substituting these limits into \eqref{eq:convergence-eq1} and taking $\varepsilon>0$ small gives
\begin{equation*}
	\limsup_{\delta \to 0}
	\frac{\phi_{\Omega^\delta}^{1}[\,a^\delta \leftrightarrow b^\delta\,]}
	{\phi_{\Omega^\delta}^{1}[\,c^\delta \leftrightarrow d^\delta\,]}
	\le
	\frac{\langle \sigma_a \sigma_b \rangle_{\Omega}^{(\mathrm{w})}}
	{\langle \sigma_c \sigma_d \rangle_{\Omega}^{(\mathrm{w})}}
	\big((1+O(\varepsilon^{c}))(1+O((\sqrt{\varepsilon}/r)^{c}))\big).
\end{equation*}

A symmetric argument holds with $\liminf_{\delta\to0}$ in place of $\limsup_{\delta\to0}$. Finally, letting $\varepsilon\to0$ yields the claimed convergence.
\end{proof}
We are now in position to prove the main result of the article.
\begin{proof}[Proof of Theorem \ref{thm:Correlation-function}]
	We first treat the case where $a^\delta,b^\delta,c^\delta,d^\delta$ are all of the same type. In this setting, one recovers the framework of \cite[Section~2.8]{CHI} and can \emph{fuse} the faces $c^\delta$ and $d^\delta$ into a single limiting point $e$ inside the domain, located at a fixed positive distance from $a$, $b$, and the boundary. The generalisation of \cite[Lemma~2.25]{CHI} follows from standard RSW-type arguments as $c^\delta$ and $d^\delta$ approach each other. Consequently, by the same reasoning as in \cite[Proof of Theorem~1.1]{CHI}, one obtains
\begin{equation*}
	\rho(\delta)^{-1}\,\phi_{\Omega^\delta}^{1}\big[\,a^\delta \leftrightarrow b^\delta\,\big]
	\longrightarrow
	\langle \sigma_a \sigma_b \rangle_{\Omega}^{+},
\end{equation*}
where
\begin{equation*}
	\rho(\delta)
	:=
	\phi_{\cS^\delta}\big[\,a^\delta \leftrightarrow (a+1)^\delta\,\big],
\end{equation*}
and $(a+1)^\delta$ has the \emph{same} type as $a^\delta$, approximating the point $a+1$ on $\cS^\delta$ up to $O(\delta)$. Finally, using the convergence of FK interfaces to SLE$(16/3)$ together with \cite{wu2018polychromatic,garban2020convergence}, one deduces that the full-plane one-arm exponent equals $\tfrac{1}{8}$, implying via standard RSW arguments that $\rho(\delta)=\delta^{-1/4+o(1)}$. When the faces are of different types, the same reasoning applies, except that in Theorem~\ref{thm:convergence-ratios}, the factors $\Phi^{\mathcal{C}_{\varepsilon}(g^\delta)}_{\mathrm{IIC}}[\,g^\delta \leftrightarrow \infty\,]$ for $g^\delta\in\{a^\delta,b^\delta,c^\delta,d^\delta\}$ no longer cancel, while the ratio is bounded away from $0$ and $\infty$, leading to the (a priori) non-trivial multiplicative prefactors stated in the theorem.
\end{proof}

\printbibliography

\end{document}